\documentclass[12pt]{amsart}
\usepackage{mathrsfs}
\usepackage{mathtools}
\usepackage{dsfont}
\usepackage{amsmath,amssymb,amsthm,upref,graphicx,mathrsfs}
\usepackage{enumerate}

\usepackage{bbm}

\usepackage{color}
\usepackage[
  colorlinks=true,
  linkcolor=blue,
  citecolor=blue,
  urlcolor=blue]{hyperref}

\usepackage{appendix}

\numberwithin{equation}{section}

\newtheorem{theorem}{Theorem}[section]
\newtheorem{proposition}[theorem]{Proposition}
\newtheorem{lemma}[theorem]{Lemma}

\theoremstyle{definition}

\newtheorem{construction}[theorem]{Construction}
\newtheorem{remark}[theorem]{Remark}

\newcommand{\M}{\mathcal{M}}
\newcommand{\1}{\mathbf{1}}
\newcommand{\Ra}{\mathcal{R}}

\newcommand{\E}{\mathbb{E}}
\newcommand{\R}{\mathbb{R}}

\newcommand{\supp}{{\rm supp}\,}
\newcommand{\be}{\begin{eqnarray*}}
\newcommand{\ee}{\end{eqnarray*}}
\newcommand{\beq}{\begin{equation}}
\newcommand{\eeq}{\end{equation}}

\newcommand{\N}{\mathcal{N}}

\begin{document}

\title[Noncommutative unbounded Vilenkin-Fourier series]
{The weak type $(1,1)$ estimate of Dirichlet Means for unbounded noncommutative Vilenkin systems}

\author[Hong]{Guixiang Hong}
\address{Institute for Advanced Study in Mathematics, Harbin Institute of Technology, Harbin 150001, China}
\email{gxhong@hit.edu.cn}

\author[Zhao]{Tiantian Zhao}
\address{Institute for Advanced Study in Mathematics, Harbin Institute of Technology, Harbin 150001, China}
\email{zhaotiantian@hit.edu.cn}




\subjclass[2020]{Primary 46L52; Secondary 42B20, 46L53}
\keywords{noncommutative Vilenkin system;   Vilenkin-Fourier series;  noncommutative $L_{p}$-spaces; noncommutative Calder\'{o}n-Zygmund decomposition}

%
%
\begin{abstract}
Let $\mathcal{R}$ be the hyperfinite $\mathrm{II}_1$ factor. Considering the partial sum operators $(\mathcal{S}_n)_{n\geq 1}$ of the noncommutative Vilenkin-Fourier series associated with an arbitrary admissible Vilenkin group, we prove that  there exists a universal  constant $c>0$ such that
\begin{equation*}	
	\sup_{n\geq1}\|\mathcal{S}_n(f)\|_{L_{1,\infty}(\mathcal{R})} \leq c\|f\|_{L_1(\mathcal{R})},\quad f \in L_1(\mathcal{R}), 
\end{equation*}
and, for every $1<p<\infty$,
$$\sup_{n\geq1}\|\mathcal{S}_n(f)\|_{L_p(\mathcal{R})} \leq c\frac{p}{p-1}\|f\|_{L_p(\mathcal{R})},\quad f \in L_p(\mathcal{R}).$$ 
Besides the transference technique, the main novel ingredient is a  modified version of noncommutative Calder\'{o}n-Zygmund decomposition  established in \cite{CCP2022}. Consequently, we resolve the problem of weak  type  $(1,1)$ estimate  communicated to the authors by Fedor Sukochev, and substantially improve  the strong type (p,p) estimates obtained in \cite{DFdePS2001} by achieving the optimal order $\frac{p}{p-1}$.
\end{abstract}

\maketitle

%
%

 \section{Introduction}
 Let $\mathbf{m}=(m_k)_{k\ge0}$ be a sequence of integers with $m_k\ge2$, and let  $G_{\mathbf{m}}=\prod_{k=0}^\infty \mathbb Z_{m_k}$ be the associated Vilenkin group equipped with the Haar probability measure.
 The Vilenkin system  $\{\phi_k\}_{k\geq0}$ (see Section \ref{sec2.4}) forms an orthonormal basis of $L_2(G_{\mathbf{m}})$. For $f\in L_1(G_{\mathbf{m}})$, the $n$-th partial sum of the Vilenkin-Fourier series is given by
 \[
 S_nf=\sum_{k=0}^{n-1}\widehat f(k)\phi_k,\quad n\geq1,
 \]
 where $\widehat f(k)$ is the $k$-th Fourier coefficient of $f$. A central theme in classical theory is the uniform boundedness of the family
 $\{S_n\}_{n\ge1}$, especially at the endpoint level.  Such bounds are intimately connected
 to mean convergence properties of Vilenkin-Fourier series. A decisive breakthrough was achieved by Young \cite{Yo1976}, who proved that,  on all commutative Vilenkin groups (in particular, allowing unbounded, i.e. $\sup_{k\geq0} m_k=\infty$), the sequence
 of partial sums of the Vilenkin-Fourier series is a uniformly bounded family of weak
 type $(1,1)$ operators and strong type $(p,p)$ operators for all $1<p<\infty$.
 Young's result is commonly viewed  as a strengthening of  earlier $L_p$ convergence results for bounded Vilenkin systems, such as those of Watari \cite{Wa1958} and Gosselin \cite{Go1973}. Moreover, it marks a genuine breakthrough from the bounded theory to the unbounded setting. For these reasons, it has since become a cornerstone of commutative Vilenkin analysis.  We refer to \cite{On1971,On1972,Pa1932,We2002,Wa1964} and the references
 therein for further developments in harmonic analysis on Vilenkin (in particular dyadic) groups. 
 

The emergence of noncommutative harmonic analysis as a broad and influential research area owes much to the foundational work of  Junge, Pisier,  and Xu \cite{Pi1997,Ju2002, JX2007} on noncommutative martingale and ergodic inequalities around the end of the last century, and one can find earlier works on the hyperfinite $\mathrm{II}_1$ factor for instance in \cite{SF1994,SF1995} by Sukochev et al..  
Chen, Xu, and Yin \cite{CXY2013} initiated harmonic analysis on quantum tori by establishing maximal inequalities and pointwise convergence results for several summation methods, identifying the completely bounded $L_p$-Fourier multipliers on quantum tori, and developing the corresponding Hardy space theory, building on Mei's operator-valued Hardy spaces \cite{Me2007}. Subsequently, Junge, Mei, and Parcet \cite{JMP2014} developed a H\"ormander-Mikhlin theory for Fourier multipliers on group von Neumann algebras associated with arbitrary discrete groups, while their later work \cite{JMP2018} provided dimension-free bounds for noncommutative Riesz transforms and further multiplier results.
More recently, Parcet, Ricard, and de la Salle \cite{PRD2022} (see also \cite{CGPT25}) established a H\"ormander-Mikhlin criterion on von Neumann algebras generated by higher-rank semisimple Lie groups.
In a related direction, Conde-Alonso, Gonz\'{a}lez-P\'{e}rez, Parcet and Tablate \cite{CGPT2023} found a  H\"ormander-Mikhlin-Schur multiplier theorem on Schatten-von Neumann classes.  Recently, Hong, Wang and Wang \cite{HWW2024} established two criteria for noncommutative maximal inequalities associated with Fourier multipliers, and applied them to the study of pointwise convergence of  noncommutative Fourier series.

The purpose of the present paper is to establish an analogue of Young's theorem in the framework of  separable hyperfinite  $\mathrm{II}_1$ factor $\mathcal R$ associated with any admissible Vilenkin group. 
It is worth stressing that, with the exception of \cite{CXY2013}, the aforementioned results are not concerned with endpoint estimates,  and more precisely, do not provide weak type $(1,1)$ bounds. 
In contrast, the main contribution of the present paper is to establish the following weak type $(1,1)$ and  strong type $(p,p)$ estimates for partial sum operators associated with noncommutative Vilenkin systems.

  Let  $\mathcal{R}$ be the hyperfinite $\mathrm{II}_1$ factor associated with an admissible Vilenkin group, and let $(\mathcal{S}_n)_{n\geq1}$  be the partial sum operators of noncommutative Vilenkin-Fourier series. For such $\mathcal{R}$, there exists a canonical normal finite faithful trace $\tau$, and the resulting $L_p$ spaces and weak $L_p$ spaces are denoted by $L_p(\mathcal{R})$ and $L_{p,\infty}(\mathcal{R})$,  respectively. We refer the reader to Sections  \ref{sec-2} and \ref{sec5}  for more information.
\begin{theorem}\label{Npb}
	There exists a universal  constant $c>0$  such that
	\begin{equation}\label{nc-w11}
		\sup_{n\geq1}\|\mathcal{S}_n(f)\|_{L_{1,\infty}(\mathcal{R})} \leq c \|f\|_{L_1(\mathcal{R})},\quad f\in L_1(\mathcal{R})
	\end{equation}
	and, for every $1<p<\infty$,
	\begin{equation}\label{nc-spp}
		\sup_{n\geq1}\|\mathcal{S}_n(f)\|_{L_p(\mathcal{R})} \leq c\frac{p}{p-1}\|f\|_{L_p(\mathcal{R})},\quad  f\in L_p(\mathcal{R}).
	\end{equation}
\end{theorem}
It was shown that the bounded Vilenkin  system forms a Schauder basis for finite von Neumann algebras in \cite{DS2000}, and this result was extended to unbounded Vilenkin groups in \cite{DFdePS2001}.  However, the corresponding bounds obtained in  \cite{DS2000,DFdePS2001} fail to attain the optimal order. Concerning  Walsh Fourier series, Jiao et al. \cite[Theorem 1.4]{JZWZ2018} proved that the sequence of partial sums forms a uniformly bounded family of weak type $(1,1)$ operators in the separable hyperfinite $\mathrm{II}_1$ factor $\mathcal{R}$.  More recently, extensions to the  noncommutative bounded Vilenkin setting have been established in \cite{SS2018,TZ2024}.  Theorem \ref{Npb} treats the case of arbitrary admissible Vilenkin groups in $\mathcal R$, and thus extends \cite{JZWZ2018,SS2018,TZ2024} from the Walsh setting to all (unbounded) Vilenkin setting. Moreover, Theorem \ref{Npb}  provides a strengthening over \cite{DFdePS2001}, by achieving the optimal order $\frac p{p-1}$. Based on the standard transference principle,  we reduce the proof of Theorem \ref{Npb} to the corresponding uniform estimates for the associated operator-valued case.  Let $\mathcal M$ be a von Neumann algebra equipped with a normal semifinite faithful trace $\tau$. Consider the tensor von Neumann algebra $\mathcal{N}=L_{\infty}(G_{\mathbf{m}})\bar{\otimes}\mathcal{M}$  equipped with the tensor trace $\varphi=\int \otimes \tau$. {Let $(S_n)_{n\geq1}$  be the partial sum operators of Vilenkin-Fourier series in $L_p(G_{\mathbf{m}})$, and the amplification operator $S_n\otimes \mathrm{id}_{L_p(\mathcal M)}$  is still denoted by $S_n$ without causing any confusion.}
\begin{theorem}\label{weak-main}
Let $G_{\mathbf{m}}$ be an arbitrary Vilenkin group. Then there exists a universal constant $c>0$ such that
\begin{equation}\label{ov-w11}
	\sup_{n\geq1}\|S_n(f)\|_{L_{1,\infty}(\mathcal{N})}\leq c\|f\|_{L_1(\mathcal{N})},\quad f\in L_1(\mathcal{N}),
\end{equation}
and, for every $1<p<\infty$,
\begin{equation}\label{ov-spp}
\sup_{n\geq1}\|S_n(f)\|_{L_p(\mathcal{N})}\leq c\frac{p}{p-1}\|f\|_{L_p(\mathcal{N})},\quad f\in L_p(\mathcal{N}).
\end{equation}
\end{theorem}
To prove Theorem~\ref{weak-main} (and hence Theorem~\ref{Npb}), a key obstruction arises from the fact that the noncommutative Calder\'{o}n-Zygmund decomposition  developed by Cadilhac, Conde-Alonso and Parcet \cite{CCP2022}  cannot  be applied  directly, since their decomposition is tailored to martingale filtrations with a finite regularity constant. More precisely, when the underlying Vilenkin group is unbounded (i.e.\ $\sup_{k\geq0} m_k=\infty$), the canonical Vilenkin filtrations is impossible to be regular, so the quantitative estimates in \cite{CCP2022} degenerate. In particular, 
one loses uniform control of the ``good'' part of the decomposition, most notably the estimate of $\|g\|_{L_\infty(\mathcal{N})}$, and hence also the crucial $L_2$ bounds needed for the endpoint weak type $(1,1)$ argument.  To overcome this difficulty, we first introduce in Subsection~\ref{fil}  a finer filtration adapted to an arbitrary Vilenkin group. Based on this filtration, we then establish in Subsection~\ref{3.2} a modified noncommutative Calder\'{o}n-Zygmund decomposition. This decomposition constitutes the key tool in our approach and may be of independent interest. Compared with the standard decomposition, the new feature is twofold. First, the decomposition produces an additional off-diagonal ``good'' part $g_{\mathrm{off}}$ besides the diagonal part $g_{\mathrm{d}}$, which is unavoidable for treating the unbounded Vilenkin geometry; moreover, we are unable to establish an $L_\infty$-estimate for $g_{\mathrm{off}}$, and instead we will  prove its $L_2$-estimate which is sufficient for the purposes. Second, the diagonal and off-diagonal bad parts $b_{\mathrm{d}}$ and $b_{\mathrm{off}}$ satisfy a new cancellation property: beyond the usual mean-zero condition on the relevant atoms, we obtain an extra cancellation against the generalized Rademacher functions (see Theorem~\ref{NewCZ}). This additional cancellation is precisely what allows the reduction to the remaining $H_k$ terms (see \eqref{S-alpha-H}) and makes the endpoint weak type $(1,1)$ estimates feasible in the unbounded setting.

The paper is organized as follows.  In next section, we introduce all required notation and
notions.   In Section \ref{sec3}, we establish a modified noncommutative Calder\'{o}n-Zygmund decomposition.  Section \ref{sec4} is devoted to proving Theorems \ref{weak-main}, and the proof of Theorem \ref {Npb} is presented in Section \ref{sec5}.

Throughout the paper, The symbol $c$ stands for a positive constant which may
vary from line to line, and we write $c_A$ to emphasize the constant  $c_A$  depends only on the
parameter$A$. For each  $a\in\mathbb{R}$, the symbols $\lfloor a \rfloor$  denotes  the greatest integer less than or equal to $a$, and $\lceil a\rceil$ denotes  the least integer greater than or equal to $a$.

\section{Preliminaries}\label{sec-2}
\subsection{Noncommutative (weak) Lebesgue spaces} 
Let $(\mathcal{M},\tau)$ be a semifinite von Neumann algebra acting on a fixed Hilbert space $H$, equipped with a faithful normal semifinite trace $\tau$. Denote by $\M^+$ the positive part of $\M$, and let $\mathcal{S}_{\M^+}$ be the set of all $x\in\M^+$ whose support projection satisfies $\tau(\supp x)<\infty$. Let $\mathcal{S}_{\M}$ be the linear span of  $\mathcal{S}_{\M^+}$. Then  $\mathcal{S}_{\M}$ is a $w^*$-dense $*$-subalgebra of $\M$.
For $1\leq p < \infty$, the noncommutative $L_p$ space associated to $\mathcal{M}$ is defined as the completion of $\mathcal{S}_{\mathcal{M}}$ under the norm
$$\|x\|_{L_p(\mathcal{M})}= \tau(|x|^p)^{1/p},\; x\in \mathcal{S}_{\mathcal{M}},$$
where $|x|=(x^*x)^{\frac{1}{2}}$ is the modulus of $x$. For $ p=\infty$, $L_{\infty}(\mathcal{M})$ is defined to be the set $\mathcal{M}$, endowed with the operator norm $\|\cdot\|_{L_{\infty}(\mathcal{M})}=\|\cdot\|_{\mathcal{M}}.$
Let $L_p^+(\M)$ denote the positive part of $L_p(\M)$.

Suppose that $\M$ acts on a separable Hilbert space $H$.
A closed densely defined operator $x$ on $H$ is said to be affiliated with
$\M$ if $u^* xu = x$ for all unitary operators $u$ in the commutant $\M'$	 of $\M$. If $x$ is a densely defined self-adjoint operator on $H$ with spectral decomposition $x=\int_{-\infty}^\infty sde_s^x$, then,
for any Borel subset $B\subset \R$, we denote by $\chi_B(x)$ the corresponding spectral projection $\int_{-\infty}^\infty \chi_B sde_s^x$.
A closed and densely defined operator $x$ affiliated with
$\M$ is called $\tau$-measurable if there exists $\lambda>0$ such that
$$\tau(\chi_{(\lambda,\infty)}(|x|))<\infty.$$
We denote the set of the $*$-algebra of $\tau$-measurable operators by $L_0(\M)$. 
For $x\in L_0({\M})$, the generalized singular value function $\mu(t,x)$ is defined by
$$\mu(t,x)=\inf\big\{s>0:\tau\big(\chi_{(s,\infty)}(|x|)\big)\leq t\big\},\quad t>0.$$
The function $t\mapsto \mu(t,x)$ is non-increasing and right-continuous; for more detailed study of the singular value function, we refer the reader {to \cite{Fa1986,DPS2023}}. Of special interest in this paper is  the noncommutative weak Lebesgue space $L_{p,\infty}(\mathcal{M})$, with the quasi-norm
$$\|x\|_{L_{p,\infty}(\M)}=\sup_{t>0}t^{1/p}\mu(t,x)=\sup_{\lambda>0}\lambda[ \tau(\chi_{(\lambda,\infty)}(|x|))]^{1/p}. $$

\subsection{Noncommutative martingales}  {Let $(\mathcal{M},\tau)$ be a finite von Neumann algebra, and let  $(\mathcal{M}_n)_{n\geq1}$ be an
increasing sequence of von Neumann subalgebras of $\mathcal{M}$ such that the union of $(\mathcal{M}_n)_{n\geq1}$  is weak-$*$ dense in $\mathcal{M}$.} For every $n\geq1$, there is a unique normal conditional expectation $\mathcal{E}_n:\mathcal{M}\to\mathcal{M}_n$ such that $\tau\circ\mathcal E_n=\mathcal E_n$. A noncommutative martingale with respect to the above filtration is a sequence $(x_n)_{n\geq1}$ in $\mathcal{M}$ such that $\mathcal{E}_{n}(x_{n+1})=x_n$ for every $n\geq1$. 
 
 Recall that $(\mathcal{M}_n)_{n\geq1}$ is a regular filtration if there exists a positive constant $R_{reg}\geq1$ such that, for each positive $x\in L_1(\mathcal M)$,
 \begin{equation}\label{reg-M}
 \mathcal{E}_{n}(x)\leq R_{reg}\mathcal{E}_{n-1}(x).
 \end{equation}

\subsection{Cuculescu's projections}
The noncommutative Calder\'{o}n-Zygmund decomposition  is based on  the Cuculescu projections.   We briefly recall  Cuculescu's construction.

\begin{lemma}[{\cite{Cuc} or \cite[Proposition 2.3]{Rand2002}}]\label{Cuculescu}
	{Let $f\in L^+_1(\mathcal{M})$. Consider the martingale $(f_n)_{n\geq1}$ with $f_n=\E_n(f)$, $n\geq1$. For any fixed $\lambda>0$,   there exists a sequence of decreasing projections  $(q_n^{(\lambda)})_{n\geq 1}$ in $\mathcal{M}$ satisfying the following properties:}
	\begin{enumerate}[{\rm (i)}]
		\item for every $n\geq1 $, $q_n^{(\lambda)}\in \mathcal{M}_n $;
		
		\item for every $n\geq1 $, $q_n^{(\lambda)}$ commutes with $q_{n-1}^{(\lambda)}f_n q_{n-1}^{(\lambda)}$;
		
		\item for every $n\geq1 $, $q_n^{(\lambda)}f_nq_n^{(\lambda)}\leq \lambda q_n^{(\lambda)}$;
		
		\item if we set $q^{\lambda}= \wedge_{n=1}^{\infty} q_n^{(\lambda)}$, then
		$\lambda\tau({\bf 1}-q^{\lambda})\leq \tau(({\bf 1}-q^{\lambda})f)\leq \|f\|_{L_1(\mathcal{N})}$.
	\end{enumerate}
\end{lemma}
In what follows, we simply write $(q_n)_{n\geq 1}$ for the sequence of Cuculescu's projections $(q_n^{(\lambda)})_{n\geq 1}$ associated with the martingale generated by  a given $f \in L_1^+(\mathcal{M})$, and write $q$ for the corresponding $q^\lambda$. Set
\begin{equation}\label{pn}
	p_n = q_{n-1}-q_n, \quad n\geq 1.
\end{equation}
Then
\begin{equation}\label{sump}
	\sum_{n\geq1} p_n={\bf 1}-q.
\end{equation}

\subsection{Vilenkin systems and the partial sum operators for operator-valued functions}\label{sec2.4}
Let $\mathbf m=(m_k)_{k\geq0}$ be a sequence of positive integers with $m_k\geq2$ for all $k\ge 0$, and 
let $\mathbb{Z}_{m_{k}}$ denote the discrete cyclic group of order $m_k$.
The associated Vilenkin group is the compact Abelian group
$$G_{\mathbf{m}}=\prod_{k=0}^{\infty}\mathbb{Z}_{m_{k}},$$
endowed with the product topology and the product Haar probability measure. In the particular case  $m_k=2$ for each $k\geq0$, the group $G_{\mathbf{m}}$ is the dyadic group. 
We say that $G_{\mathbf{m}}$ is bounded (resp. unbounded) if the sequence
$ m=(m_k)_{k\ge 0}$ is bounded (resp. unbounded).

Throughout the paper, we identify  $(G_{\mathbf{m}},+_{G_{\mathbf{m}}})$ with $([0,1),\dotplus)$. Set $M_0:=1$ and $M_{k+1}:=m_k M_k$ for $k\ge 0$. Then every $x\in[0,1)$ admits an expansion
\begin{equation}\label{x-expre}
x=\sum\limits_{k=0}^\infty\frac{x_k}{M_{k+1}},\quad\quad 0\leq x_k < m_k,\quad x_k \in \mathbb{N}.
\end{equation}
If $x, y\in [0,1)$ are given by
$$x=\sum_{k=0}^\infty \frac{x_k}{M_{k+1}} ,\quad y=\sum_{k=0}^\infty\frac{y_k}{M_{k+1}}, \quad0\leq x_k, y_k < m_k,\quad x_k, y_k \in \mathbb{N},$$
we define
$$x\dotplus y:=\sum_{k=0}^\infty\frac{(x_k \oplus y_k)}{M_{k+1}}, \quad\quad 0\leq x_k, y_k < m_k,\quad x_k, y_k \in \mathbb{N},$$
where
$$x_k\oplus y_k=(x_k +y_k)\,\mbox{mod}\,m_k. $$
We denote by $\dot{-}$ the inverse operation of $\dot{+}$.

The generalized Rademacher functions are defined by
$$
r_k(x):=\exp\big(\frac{2\pi \mathrm{i}x_k}{m_k}\big), \qquad x\in [0,1),\qquad k\geq0.
$$
The product system generated by the generalized Rademacher functions is the  Vilenkin system:
$$
\phi_{n}:=\prod_{k=0}^{\infty}{r_k}^{n_k},\qquad n\geq1,
$$
where $n$ has the representation
\begin{equation}\label{e25}
n=\sum_{k=0}^{\infty} n_k M_{k}, \qquad  0 \leq n_k < m_k,    n_k \in \mathbb N.
\end{equation}
For the natural numbers $n=\sum_{k=0}^{\infty}n_kM_{k}$ and $\ell=\sum_{k=0}^{\infty} \ell_kM_{k}$, we define
\begin{equation}\label{addition}
n\dotplus \ell:=\sum_{k=0}^{\infty}(n_k \oplus \ell_k)M_{k}.
\end{equation}
The following elementary properties can be found in  \cite{Gat1999}.
\begin{lemma} \label{propetyVilenkin}
 Let $n,\ell \in \mathbb{N}$, and $\overline n=\sum_{k=0}^{\infty}(m_k-n_k)M_{k}$. We have
\begin{enumerate}[{\rm (i)}]
\item $|\phi_{n}|=1,$
\item $\phi_{n\dotplus \ell}=\phi_n \phi _\ell,$
\item $\phi_{\overline n}=\overline{\phi}_n$.
\end{enumerate}
\end{lemma}

In the sequel, we let $\mathcal{N}=L_{\infty}(0,1)\bar{\otimes }\mathcal{M}$ equipped with the  tensor trace $\varphi=\int \otimes \tau$.  For $f \in L_1(\mathcal{N})$, the $k$-th  Vilenkin-Fourier coefficient of $f$ is defined by
$$\widehat {f}(k) := \int^1_0 f(x) \overline{\phi}_{k}(x)d\mu(x), \qquad k\geq0.$$
Now $\widehat{f}(k)\in \mathcal{M}$. If  $\mathcal{M}=\mathbb{C}$, then $\widehat{f}(k)$ is a number. Denote by $S_{n}(f)$ the $n$-th partial sum of the Vilenkin-Fourier series of  $f\in L_1(\mathcal{N})$, namely,
\begin{equation}\label{ps}
S_{n}(f) := \sum_{k=0}^{n-1} \widehat {f}(k)\phi_{k},\qquad n\geq1.
\end{equation}
We can rewrite it as follows
$$S_{n} f(x) =f\ast D_n(x) =\int_0^1 f(y) D_{n}(x\dot{-}y) d\mu(y), \qquad  x\in[0,1),$$
where $D_{n}$ is the Vilenkin Dirichlet kernel defined by
$$
D_n := \sum_{k=0}^{n-1} \phi_k.
$$
Recall  that the  Vilenkin-Dirichlet kernels (see e.g. \cite[page 312]{Yo1976}) satisfy
\begin{equation}\label{e5}
	D_{M_n}(x) = \Big\{
	\begin{array}{ll}
		M_n, & \hbox{if $x \in [0,\frac{1}{M_n})$} \\
		0, & \hbox{if $x \in [\frac{1}{M_n},1)$}
	\end{array}.
	\Big.
\end{equation}
It follows from \eqref{e5} that
\begin{equation}\label{S2n}
S_{M_n}f(x)=\int_0^1 f(y) D_{M_n}(x\dot{-}y) d\mu(y)=\mathbb{E}_n(f)(x), \qquad n\geq1, x\in [0,1).
\end{equation}

As a consequence of the orthogonality of Vilenkin system, we obtain the following elementary results.
\begin{lemma}\label{lem:type22}
	If $f\in L_2(\mathcal{N})$, then
	$$\sup_{n\geq1}\|S_n(f)\|_{L_2(\mathcal{N})}\leq \|f\|_{L_2(\mathcal{N})},$$
	and moreover,
	$$\sum_{k=0}^\infty\varphi(|\widehat f(k)|^2)= \|f\|_{L_2(\mathcal{N})}^2 .$$
\end{lemma}

\section{A modified  noncommutative Calder\'{o}n-Zygmund decomposition}\label{sec3}
In this section, we establish a modified Calder\'{o}n-Zygmund decomposition based on Cuculescu's construction for a standard noncommutative martingale.  Accordingly, we first introduce a finer filtration on an arbitrary Vilenkin Group.

\subsection{A finer filtration on Vilenkin groups}\label{fil}
We start with a review of the canonical filtration on the Vilenkin group. Consider the probability space $([0,1),\mathcal{F}',\mu)$, where $\mu$ denotes the Lebesgue measure on $[0,1)$ and $\mathcal{F}'$ is the $\sigma$-algebra  generated by the Vilenkin intervals
$$\mathcal{F}':=\sigma\Big\{I_{k}(x): k \in \mathbb{N},  x\in [0,1)\Big\},$$
where, under the identification of $G_{\mathbf m}$ with $[0,1)$, we define for each $x\in [0,1)$ and $k\in \mathbb{N}$, 
$$I_{0}(x):=G_{\mathbf{m}}, \quad I_{k}(x):=\Big\{t=(t_{i})_{i \in \mathbb{N}} \in G_{\mathbf{m}}: t_{i}=x_{i} \text { for } i<k\Big\},$$
each $I_k(x)$ admits the interval representation
$$I_k(x)=[\frac{j}{M_k}, \frac{j+1}{M_k}),$$
where $0\leq j<M_k$ is  uniquely determined by the relationship $x\in [\frac{j}{M_k}, \frac{j+1}{M_k})$.
Then $\{I_k(x):k\ge0\}$ forms a neighborhood base at $x$. In particular, we write $I_k:=I_k(0)$.   For each $k\geq 0$, set
\begin{equation}\label{VFn}
	\mathcal{F}'_k=\sigma \{ I_k(x): x\in [0,1)\}.
\end{equation}
Then $(\mathcal{F}'_k)_{k\geq0}$ is a filtration of $\mathcal{F}'$.  Let $F_{k}'$ denote the collection of all atoms of $\mathcal{F}_{k}'$.
Since this canonical filtration fails to be regular on any unbounded Vilenkin group, we introduce below a finer regular filtration.

\begin{construction}\label{construction} Let $\mathbf{m}=(m_k)_{k\geq0}\subset\mathbb{N}$ with $m_k\geq2$ for all $k\geq0$. Consider the probability space $([0,1),\mathcal F',\mu)$ introduced above. We construct a new $\sigma$-algebra $\mathcal F$ together with a two-parameter filtration
	\[
	(\mathcal F_{k,\ell})_{k\ge0,\ 1\le \ell\le \lceil\log_2 m_{k-1}\rceil},
	\]
which generates $\mathcal F$, where we set $m_{-1}=2$, that is $\lceil\log_2 m_{k-1}\rceil=1$, and $\mathcal F_{0,1}:=\sigma\{\emptyset,[0,1)\}=\mathcal F_0'$ is the $\sigma$ trivial algebra.  Let $I_0=[0,1)$.
 
\smallskip
\noindent\textbf{Step 1 (level $k=1$, with $M_1=m_0$).}  Divide $I_0$ into two $ \mathcal{F}'_1$-measurable subintervals
$$I_{1,1}^1=[0,\frac{\lfloor\frac{m_0}{2}\rfloor}{M_1}), \quad I_{1,1}^2=[\frac{\lfloor\frac{m_0}{2}\rfloor}{M_1},1),$$
and define
$$\mathcal{F}_{1,1}=\sigma\{I_{1,1}^1, \quad I_{1,1}^2\}.$$
These intervals satisfy
$$\mu(I_{1,1}^2)-M_1^{-1}\leq\mu(I_{1,1}^1)\leq \mu(I_{1,1}^2).$$ 
Throughout the construction, whenever an interval is split into two subintervals, we require the left subinterval to have measure no larger than  the right one.
 If  both $I_{1,1}^1$ and $I_{1,1}^2$  are  atoms of $\mathcal{F}'_1$,  then  $\mathcal{F}_{1,1}=\mathcal{F}'_1$ and the procedure stops at this stage.  Otherwise, every interval that is not yet an atom of $\mathcal F_1'$ is subdivided again in the same non-equal manner. For example:
\begin{enumerate}[{\rm (i)}]
\item If $I_{1,1}^1$ is an atom while $I_{1,1}^2$ is not, we subdivide $I_{1,1}^2$ as:  
	$$I_{1,1}^2=I_{1,2}^3\cup I_{1,2}^4,\quad I_{1,2}^3, I_{1,2}^4\in\mathcal{F}'_1,$$
	where $$I_{1,2}^3=[\frac{\lfloor\frac{m_0}{2}\rfloor}{M_1},\frac{\lfloor\frac{m_0}{2}\rfloor+\lfloor\frac{m_0-\big\lfloor\frac{m_0}{2}\rfloor}{2}\big\rfloor}{M_1})\quad\mbox{and}\quad I_{1,2}^4=[\frac{\lfloor\frac{m_0}{2}\rfloor+\lfloor\frac{m_0-\lfloor\frac{m_0}{2}\rfloor}{2}\rfloor}{M_1},1).$$
In this case, we set
	$$\mathcal{F}_{1,2}=\sigma\{I_{1,2}^1=I_{1,2}^2:=I_{1,1}^1, I_{1,2}^3, I_{1,2}^4\}.$$

\item If neither  $I_{1,1}^1$ nor $I_{1,1}^2$  is an atom of $\mathcal{F}'_1$, we additionally subdivide $I_{1,1}^1$:
	$$I_{1,1}^1=[0,\frac{\lfloor\frac{m_0}{2^2}\rfloor}{M_1})\cup[\frac{\lfloor\frac{m_0}{2^2}\rfloor}{M_1},\frac{\lfloor\frac{m_0}{2}\rfloor}{M_1})=:I_{1,2}^1\cup I_{1,2}^2,$$
	 with $I_{1,2}^1,I_{1,2}^2\in\mathcal F'_1$, and we define
	$$\mathcal{F}_{1,2}=\sigma\{I_{1,2}^1, I_{1,2}^2, I_{1,2}^3, I_{1,2}^4\}.$$
\end{enumerate}
Each newly created pair satisfies the same  measure constraint 
$$\mu(I_{1,2}^2)-M_1^{-1}\leq\mu(I_{1,2}^1)\leq \mu(I_{1,2}^2)\quad\mbox{and\quad $\mu(I_{1,2}^4)-M_1^{-1}\leq\mu(I_{1,2}^3)\leq \mu(I_{1,2}^4)$}.$$ 
After at most $\lceil\log_2 m_0\rceil$  such steps, we reach $\mathcal{F}'_1$, that is
$$\mathcal{F}_{1,\lceil\log_2 m_0\rceil}=\mathcal{F}'_1.$$

\smallskip
\noindent\textbf{Step 2 (level $k=2$, with $M_2=m_0m_1$).}
For each atom $I_{1,\lceil\log_2 m_0\rceil}^i\in \mathcal{F}_{1,\lceil\log_2 m_0\rceil}$, we apply the same non-equal splitting procedure. This gives rise to subintervals
$$I_{2,1}^{2i-1},\ I_{2,1}^{2i},\quad1\leq i\leq m_0,$$
satisfying
$$\mu(I_{2,1}^{2i})-M_2^{-1}\leq\mu(I_{2,1}^{2i-1})\leq\mu(I_{2,1}^{2i}).$$
Define
$$\mathcal{F}_{2,1}=\sigma\{I_{2,1}^{2i-1},I_{2,1}^{2i}:\quad 1\leq i\leq m_0\}.$$
If  all these intervals are already atoms  of $\mathcal{F}'_2$, then $\mathcal F_{2,1}=\mathcal F_2'$; otherwise, we continue the same refinement for at most $\lceil\log_2 m_1\rceil$  steps until 
$$\mathcal{F}_{2,\lceil\log_2 m_1\rceil}=\mathcal{F}'_2.$$
Proceeding inductively, we obtain a sequence of $\sigma$-algebras with two parameters
$$(\mathcal{F}_{k,\ell})_{k\geq0, 1\leq \ell\leq\lceil\log_2 m_{k-1}\rceil}.$$
From the above construction, it is easy to conclude that there is a natural   order on these algebras, and thus they form an increasing filtration. Finally, set
\[
\mathcal F:=\sigma\bigl(\bigcup_{k\geq0}\ \bigcup_{\ell=1}^{\lceil\log_2 m_{k-1}\rceil}\mathcal F_{k,\ell}\bigr)\quad \mbox{and}\quad F:=\bigcup_{k\geq0}\ \bigcup_{\ell=1}^{\lceil\log_2 m_{k-1}\rceil} F_{k,\ell},
\]
where   $F_{k,\ell}$ denote the collection of all atoms of $\mathcal{F}_{k,\ell}$.

For $f\in L_\infty([0,1),\mathcal F,\mu)$, let $(\mathbb{E}_{k,\ell})_{k\geq0,1\leq \ell\leq\lceil\log_2 m_{k-1}\rceil}$ denote  the conditional expectations with respect to the filtration $(\mathcal{F}_{k,\ell})_{k\geq0, 1\leq \ell\leq\lceil\log_2 m_{k-1}\rceil}$. They are given by
\begin{equation}\label{condexp}
\mathbb{E}_{k,\ell}(f)=\sum_{Q\in F_{k,\ell}} f_Q \chi_Q, \quad  k\geq0, 1\leq \ell\leq\lceil\log_2 m_{k-1}\rceil,
\end{equation}
where 
$$f_Q=\frac{1}{|Q|}\int_{Q} f(x)d\mu(x).$$
For each $k,\ell$, we simply denote $L_{\infty}([0,1),\mathcal{F}_{k,\ell},\mu)$ by $ L_{\infty}(\mathcal{F}_{k,\ell})$, and set
\begin{equation}\label{VNA}
\mathcal{N}_{k,\ell}=L_{\infty}(\mathcal{F}_{k,\ell})\bar{\otimes}\mathcal{M}.
\end{equation}
Then $(\mathcal{N}_{k,\ell})_{k\geq0,1\leq \ell\leq\lceil\log_2 m_{k-1}\rceil}$ is an increasing sequence of von Neumann subalgebras of $\mathcal{N}$ such that $\cup_{k\geq0,1\leq \ell\leq\lceil\log_2 m_{k-1}\rceil}\mathcal{N}_{k,\ell} $ is weak$^{\ast}$ dense in $\mathcal{N}$. The conditional expectation of $\mathcal{N}$ onto $\mathcal{N}_{k,\ell}$ is $\mathbb{E}_{k,\ell}\otimes \mathrm{id}_{\mathcal{M}}$, simply denoted by $\mathbb{E}_{k,\ell}$ whenever there is no confusion, where $\mathrm{id}_{\mathcal{M}}$ is the identity map on $\mathcal{M}$. 
\end{construction}

\begin{lemma}\label{regular} The filtration $(\mathcal{N}_{k,\ell})_{k\geq0,1\leq \ell\leq\lceil\log_2 m_{k-1}\rceil}$ is regular, with constant $R_{\rm reg}\leq3$. 
\end{lemma}
\begin{proof} Since
	\[\mathcal N_{k,\ell}=L_\infty(\mathcal F_{k,\ell})\bar\otimes \mathcal M,\]
it suffices to verify the corresponding regularity estimate for the underlying scalar filtration
	\[
	(\mathcal F_{k,\ell})_{k\ge 0,\ 1\le \ell\le \lceil \log_2 m_{k-1}\rceil}.
	\]
We view the doubly indexed family $(\mathcal F_{k,\ell})$ in natural  order and adopt the convention
\[
\mathcal F_{k,\ell-1} =
	\mathcal F_{k-1,\lceil \log_2 m_{k-2}\rceil},\  \text{if } \  \ell=1, 
\]
Thus, for every non-initial index $(k,\ell)$,  fix \(k\ge 1\), \(1\le \ell\le \lceil \log_2 m_{k-1}\rceil\), and let \(Q\in F_{k,\ell}\). Here and in the sequel,  we denote by $\widehat Q\in F_{k,l-1}$ the unique atom containning $Q$.  We distinguish two cases.
	
	\medskip
	\noindent\textbf{Case 1. \(\widehat Q\) is not subdivided in passing from \(\mathcal F_{k,\ell-1}\) to \(\mathcal F_{k,\ell}\).} In this case, we have
	\[
Q=\widehat Q,\quad\mbox{and}\quad	|\widehat Q|=|Q|.
	\]
	
	\medskip
	\noindent\textbf{Case 2. \(\widehat Q\) is subdivided in passing from \(\mathcal F_{k,\ell-1}\) to \(\mathcal F_{k,\ell}\).} Then \(\widehat Q\) is split into two children, say \(Q^l\) and \(Q^r\), by construction,
	\[
	|Q^r|-M_k^{-1}\le |Q^l|\le |Q^r|.
	\]
	Moreover, each atom of \(\mathcal F_{k,\ell}\) is a union of atoms of \(\mathcal F_k'\). Consequently,
	\[
	|Q^l|\ge M_k^{-1},
	\qquad
	|Q^r|\ge M_k^{-1}.
	\]
	We now consider the two possible positions of \(Q\).
	
	\smallskip
	\noindent\textbf{Subcase 2.1. \(Q=Q^l\).}
	Then
	\[
	|\widehat Q|
	=
	|Q^l|+|Q^r|
	\le
	|Q^l|+\bigl(|Q^l|+M_k^{-1}\bigr)
	\le
	3|Q^l|
	=
	3|Q|.
	\]
	
	\smallskip
	\noindent\textbf{Subcase 2.2. \(Q=Q^r\).}
	Then
	\[
	|\widehat Q|
	=
	|Q^l|+|Q^r|
	\le
	2|Q^r|
	=
	2|Q|.
	\]
	Thus, in all cases,
	\[
	|Q|\le |\widehat Q|\le 3|Q|.
	\]
	This proves that the filtration is regular with constant \(R_{\mathrm{reg}}\le 3\).
\end{proof}

Let $f\in L_1^+(\mathcal N)$, $\lambda>0$  and  $(\mathcal{N}_{k,\ell})_{k\geq0,1\leq \ell\leq\lceil\log_2 m_{k-1}\rceil}$ be the filtration introduced in Construction \ref{construction}.
Denote $f_{k,\ell}=\mathbb{E}_{k,\ell}(f)$.  Applying Cuculescu's construction \cite{Cuc} to $(f_{k,\ell})_{k\geq0,1\leq \ell\leq\lceil\log_2 m_{k-1}\rceil}$, one can find a sequence of decreasing projections
$(q_{k,\ell})_{k\geq0,1\leq \ell\leq\lceil\log_2 m_{k-1}\rceil}$. That is, one defines inductively for $k\geq0,1\leq \ell\leq\lceil\log_2 m_{k-1}\rceil$,  
$$q_{k,l}:=\chi_{[0,\lambda]}(\widehat{q}_{k,\ell}f_{k,\ell}\widehat{q}_{k,\ell}),$$
where
\[
\widehat{q}_{k,\ell}=
\begin{cases}
	\mathbf{1}_{\mathcal{N}} & \text{if } k=0\ \mbox{and}\ \ell=1, \\
	q_{k-1,\lceil\log_2 m_{k-2}\rceil} & \text{if } k\geq1\ \mbox{and}\ \ell=1,\\
	q_{k,\ell-1} &\text{otherwise}.
\end{cases}
\]
Then in the present semi-commutative setting, i.e. $\mathcal{N}=L_{\infty}(0,1)\bar{\otimes }\mathcal{M}$, each $q_{k,\ell}$ admits the following expression
\begin{equation}\label{qnlQ}
	q_{k,\ell}
	=\sum_{Q\in F_{k,\ell}} q_Q\chi_Q,
\end{equation}
where $q_Q$ is a projection in  $\mathcal{M}$ given by $q_Q =\chi_{(0,\lambda]}(q_{\widehat{Q}} f_Q q_{\widehat{Q}})$ with $q_{\widehat{Q}}=1_\mathcal M$ if $Q=[0,1)$.
Then one can define the sequence $(p_{k,\ell})_{k\geq0,1\leq \ell\leq\lceil\log_2 m_{k-1}\rceil}$ of disjoint projections by 
\begin{equation}\label{pnlQ}
	p_{k,\ell}=\widehat{q}_{k,\ell} -q_{k,\ell}=\sum_{Q\in F_{k,\ell}}(q_{\widehat{Q}}-q_Q)\chi_Q=\sum_{Q\in F_{k,\ell}}p_Q\chi_Q.
\end{equation}
Moreover,
\begin{equation}\label{1-q}
\mathbf 1_{\mathcal N}-q=\sum_{k\geq0}\sum_{\ell=1}^{\lceil\log_2 m_{k-1}\rceil}p_{k,\ell}=\sum_{k\geq0}\sum_{\ell=1}^{\lceil\log_2 m_{k-1}\rceil}\sum_{Q\in F_{k,\ell}}p_Q\chi_Q=\sum_{k\geq-1}\sum_{\ell=1}^{\lceil\log_2 m_{k}\rceil}\sum_{Q\in F_{k+1,\ell}}p_Q\chi_Q.
\end{equation}

The following lemma, taken from \cite[Proposition~1.1]{Me2007}, will be used repeatedly in the sequel.
\begin{lemma}\label{AcB}
For any $a,b\in L_2(\mathcal{N})$, there exists $u\in L_{\infty}(\M)$ with $\|u\|_{L_{\infty}(\M)}\leq 1$ such that
\[
\int_0^1 a^{*}(x)b(x)d\mu(x)=\Big(\int_0^1 a^{*}(x)a(x)d\mu(x)\Big)^{1/2} \cdot u\cdot\Big(\int_0^1 b^{*}(x)b(x)d\mu(x)\Big)^{1/2}.
\]
\end{lemma}

\subsection{A modified  noncommutative Calder\'{o}n-Zygmund decomposition}\label{3.2}
We are now in a position to establish the following new Calder\'{o}n-Zygmund decomposition adapted to the regular filtration constructed above.

\begin{theorem}\label{NewCZ} \rm
  Let $f\in L_1^+(\mathcal{N})$, $\lambda>0$ and  $n\in\mathbb{N}$. Write $n=\Sigma_{k=0}^{\infty} \alpha_{k} M_{k}$ with $0 \leqslant \alpha_{k}<m_{k}.$
  Let $
  (q_{k+1,\ell})_{k\ge-1,\ 1\le \ell\le \lceil\log_2 m_{k}\rceil}$ and 
 $ (p_{k+1,\ell})_{k\ge -1,\ 1\le \ell\le \lceil\log_2 m_{k}\rceil}$  be the two
 sequences of projections appearing in \eqref{qnlQ} and \eqref{pnlQ} associated with $f$ and $\lambda$, with the convention that $m_{-1}=2$. Then there  exists a projection $\zeta\in \mathcal{P}(\N)$, defined by 
\begin{equation}\label{zeta}
\zeta=\big(\bigvee_{Q\in F}p_Q\chi_{3Q}\big)^{\perp},
\end{equation}
{where \(F=\bigcup_{k\geq0}\ \bigcup_{\ell=1}^{\lceil\log_2 m_{k-1}\rceil} F_{k,\ell}\),  $F_{k,\ell}$ denote the collection of all atoms of $\mathcal{F}_{k,\ell}$,  }and $3Q$ denotes  the $3$-fold concentric dilation of $Q$, and a decomposition of $f$,
	\begin{equation}\label{dec}
		f=g_{\mathrm{d}}+g_{\mathrm{off}}+b_{d}+b_{\mathrm{off}},
	\end{equation}
such that the following assertions hold.
	\begin{enumerate}[\rm (i)]
        \item $\lambda\varphi(\mathbf{1}-\zeta)\leq\varphi((\mathbf{1}-\zeta)f)\leq c\|f\|_{L_1(\mathcal{N})}$, recall that $\varphi = \int \otimes \tau$.
        
\item $g_{\mathrm{d}}=qfq+\sum_{k\geq -1}\sum_{\ell=1}^{\lceil\log_2 m_{k}\rceil}\sum_{Q\in F_{k+1,\ell}} p_QA_Q(f)p_Q\chi_Q$; moreover, $\|g_{\mathrm{d}}\|_{L_1(\mathcal{N})}\leq c \|f\|_{L_1(\mathcal{N})}$ and $\|g_{\mathrm{d}}\|_{L_{\infty}(\mathcal{N})}\leq c\lambda$. Here, for each $Q\in F_{k+1,\ell}$, the function $A_Q(f):Q\to L_1(\mathcal{M})$ is defined as follows
		\begin{equation}\label{AQ1}
			A_Q(f)(x)=\frac{1}{|Q|}\int_Qf(y) K_{A_Q}(x,y)d\mu(y),
		\end{equation}
where we set $K_{A_Q}(x,y)=1$ if $Q\in F_{k+1}'$ or $\alpha_k=0$, and otherwise
\begin{equation}\label{AQ}
	K_{A_Q}(x,y)=\frac{\frac{1}{|Q|}\int_Q\big(r_k^{\alpha_k}(y)-r_k^{\alpha_k}(t)\big)\big(r_k^{-\alpha_k}(x)-r_k^{-\alpha_k}(t)\big)d\mu(t)}{1-\big|\frac{1}{|Q|} \int_{Q} r_k^{\alpha_{k}} (x)d \mu(x)\big|^{2}}.
\end{equation}

       \item $g_{\mathrm{off}}=\sum_{k\geq -1}\sum_{\ell=1}^{\lceil\log_2 m_{k}\rceil}\sum_{Q\in F_{k+1,\ell}} (p_QA_Q(f)q_Q\chi_Q+q_QA_Q(f)p_Q\chi_Q)$, and 
       \begin{equation}\label{goff-2-lambda f}
       	\|g_{\mathrm{off}}\|_{L_{2}(\mathcal{N})}^2\leq c\lambda\|f\|_{L_1(\mathcal{N})}.
       \end{equation}
       
		\item $b_{\mathrm{d}}=\sum_{k\geq -1}\sum_{\ell=1}^{\lceil\log_2 m_{k}\rceil}\sum_{Q\in F_{k+1,\ell}} b_{\mathrm{d}}^{k+1,\ell,Q}$, where
\begin{equation}\label{bdk+1lQ}
b_{\mathrm{d}}^{k+1,\ell,Q}=p_Q(f-A_Q(f))p_Q\chi_Q
\end{equation}
satisfies the cancellation conditions: for each $Q\in F_{k+1,\ell}$,
\begin{equation}\label{bd0}
\int_Qb_{\mathrm{d}}^{k+1,\ell,Q}(x)d\mu(x)=0\quad\mbox{and }\quad \int_Qb_{\mathrm{d}}^{k+1,\ell,Q}(x)r_k^{\alpha_k}(x)d\mu(x)=0;
\end{equation}
furthermore,
 $$\sum_{k\geq -1}\sum_{\ell=1}^{\lceil\log_2 m_{k}\rceil}\sum_{Q\in F_{k+1,\ell}}\|b_{\mathrm{d}}^{k+1,\ell,Q}\|_{L_1(\mathcal{N})}
       \leq c\|f\|_{L_1(\mathcal{N})}.$$

\item $b_{\mathrm{off}}=\sum_{k\geq -1}\sum_{\ell=1}^{\lceil\log_2 m_{k}\rceil}\sum_{Q\in F_{k+1,\ell}}b_{\mathrm{off}}^{k+1,\ell,Q}$, where
\begin{equation}\label{boffk+1lQ}
b_{\mathrm{off}}^{k+1,\ell,Q}=p_Q(f-A_Q(f))q_Q\chi_Q+q_Q(f-A_Q(f))p_Q\chi_Q
\end{equation} 
satisfies the cancellation conditions: for each $Q\in F_{k+1,\ell}$,
\begin{equation}\label{boff0}
\int_Qb_{\mathrm{off}}^{k+1,\ell,Q}(x)d\mu(x)=0\quad \mbox{and}\quad \int_Qb_{\mathrm{off}}^{k+1,\ell,Q}(x)r_k^{\alpha_k}(x)d\mu(x)=0.
\end{equation}
\end{enumerate}
\end{theorem}

\begin{remark}\label{AQ-cho-b} \begin{enumerate}[\rm (i)] 
\item  A key feature of our Calder\'{o}n-Zygmund decomposition is the specific choice of  $A_Q(f)$. This choice   ensures that the resulting bad terms satisfy the reinforced cancellation properties \eqref{bd0} and \eqref{boff0}: on $Q$, besides being orthogonal to constant functions, $b^{k+1,\ell,Q}_d$ is also orthogonal  to the relevant frequency $r_k^{\alpha_k}$.  The choice of $A_Q(f)$ in \eqref{AQ1} and \eqref{AQ} is precisely intended to ensure these orthogonality conditions. This plays an essential role in the proof.  In particular, substituting $b_{\mathrm d}$ and $b_{\mathrm{off}}$ into the decomposition \eqref{S-alpha-H}, the contributions of the first three terms vanish, and the estimates reduce to controlling the terms involving $H_k$; one may find the the details in the arguments for Lemma \ref{bd-w11} and Lemma \ref{boff-w11}.

\item In contrast with  $g_{\mathrm{d}}$, we cannot directly obtain the  \(L_\infty(\mathcal N)\)-estimate for \(g_{\mathrm{off}}\).  But we are able to establish   the $L_2(\mathcal{N})$-estimate \eqref{goff-2-lambda f},
 which is sufficient for the proofs of Theorem \ref{Npb} and \ref{weak-main}.
\end{enumerate}
\end{remark}

{The following lemma is needed for the later proof of Theorem \ref{NewCZ}.}
\begin{lemma}\label{KAQ} Let $Q\in F_{k+1,\ell}$, where $k\geq -1$, $1\leq \ell\leq\lceil\log_2 m_{k}\rceil$. Then there exists a constant $c>0$, independent of $Q$, $k$ and $\ell$, such that 
	$$|K_{A_Q}(x,y)|\leq c,\quad x,y\in Q,$$
{where \(K_{A_Q}\) is given in \eqref{AQ}.}
\end{lemma}

\begin{proof} If $Q\in F_{k+1}'$ or $\alpha_k=0$,  then
	$$K_{A_Q}(x,y)=1, \quad x,y\in Q.$$
	It therefore remains to consider the nontrivial case  $Q\notin\ F_{k+1}'$ and $\alpha_k\neq 0$.  Set $$\theta=\frac{2\pi\alpha_k}{m_k},\quad-2\pi\leq \theta\leq 2\pi.$$
	 We claim that 
	\begin{equation}\label{aq-c1}
	|r_k^{\alpha_k}(x)-r_k^{\alpha_k}(y)|
	\le
	\min\{2,b|\theta|\},\quad x,y\in Q.
	\end{equation}
	Indeed, we first have the trivial bound
	$$|r_k^{\alpha_k}(x)-r_k^{\alpha_k}(y)|\leq 2.$$
	On the other hand, we recall that  every $x,y\in[0,1)$ admit the expansions as in \eqref{x-expre}, namely
	$$x=\sum_{j=0}^\infty \frac{x_j}{M_{j+1}} ,\quad y=\sum_{j=0}^\infty\frac{y_j}{M_{j+1}}, \quad0\leq x_j, y_j < m_j,\quad x_j, y_j \in \mathbb{N}.$$
	Without loss of generality, we assume that $Q=\Big[\frac{a}{M_{k+1}},\frac{a+b}{M_{k+1}}\Big),$ where $a,b \in \mathbb{N}$ and $b=|Q|M_{k+1}\geq2$, since  $Q\notin F_{k+1}'$.  Then, for $x,y\in Q$, 
	\[
	|r_k^{\alpha_k}(x)-r_k^{\alpha_k}(y)|
	=
	\bigl|e^{i\theta x_k}-e^{i\theta y_k}\bigr|
	\le|\theta|\,|x_k-y_k|
	\le b|\theta|.
	\]
This proves the claim.
	
A straightforward calculation gives
\begin{equation}\label{aq-c2}
\begin{aligned}
	\Big| \frac{1}{|Q|} \int_{Q} r_k^{\alpha_k} (x)\, d\mu(x) \Big|&=\Big|\frac{1}{|Q|}\sum_{\ell=a}^{a+b-1}\int_{\frac{\ell}{M_{k+1}}}^{\frac{\ell+1}{M_{k+1}}}\exp(i\theta x_k)d\mu(x)\Big|\\
	&=\Big|\frac{1}{|Q|M_{k+1}}\sum_{\ell=a}^{a+b-1}\exp(i\theta \ell)\Big|\\
	&= \Big|\frac{1-\exp(ib\theta)}{b(1-\exp(i\theta))}\Big|=\Big|\frac{\sin(b\theta/2)}{b\sin(\theta/2)}\Big|=:	R_b(\theta).
\end{aligned}
\end{equation}
Recalling the definition of $K_{A_Q}(x,y)$ in \eqref{AQ}, and using  \eqref{aq-c1} and \eqref{aq-c2}, we have
\begin{equation}\label{KAQ-sin}
	\begin{aligned}
		|K_{A_Q}(x,y)| &=\Big|\frac{\frac{1}{|Q|}\int_Q\big(r_k^{\alpha_k}(y)-r_k^{\alpha_k}(t)\big)\big(r_k^{-\alpha_k}(x)-r_k^{-\alpha_k}(t)\big)d\mu(t)}{1-\big|\frac{1}{|Q|} \int_{Q} r_k^{\alpha_{k}}(x)d \mu(x)\big|^{2}}\Big|\\
		&\leq\min\{4,(b\theta)^2\}\,\bigl(1-R_b(\theta)^2\bigr)^{-1}.
	\end{aligned}
\end{equation}

It suffices to consider $|\theta|\leq\pi$,  the remaining case follows similarly.

\smallskip
\noindent\textbf{Case 1: $b|\theta|\leq\frac{\pi}{10}$.} By the Taylor expansion $\sin^2x=x^2-\frac{x^4}{3}+O(x^6)$ and the fact that $b\geq2$, we obtain
\begin{align*}
	|K_{A_Q}(x,y)| &\leq (b\theta)^2\times \Big[\frac{b^2[\frac{\theta^2}{4}-\frac{\theta^4}{48}+O(\theta^6)]-[\frac{b^2\theta^2}{4}-\frac{b^4\theta^4}{48}+O(b^6\theta^6)]}{\frac{b^2\theta^2}{4}+O(b^2\theta^4)}\Big]^{-1}\\
	&=(b\theta)^2\times \Big(\frac{\frac{b^2\theta^2}{4}+O(b^2\theta^4)}{\frac{b^4\theta^4}{48}(1-\frac{1}{b^2})+O(b^6\theta^6)}
	\Big)\\
	&=\frac{\frac{1}{4}+O(\theta^2)}{\frac{1}{48}(1-\frac{1}{b^2})+O(b^2\theta^2)}\leq \frac{\frac{1}{4}+\pi^2}{\frac{1}{48}(1-\frac{1}{4})}=16+64\pi^2.
\end{align*}
	
\smallskip
\noindent\textbf{Case 2: $b|\theta|\geq\frac{\pi}{10}$.}
	It suffices to show that there exists a constant $0<\delta<1$ such that
	$$R_b(\theta)<\delta.$$
	Since $R_b$ is an even function, we only need to consider $\theta\in[0,\pi]$. Set
	\[
	u:=\frac{\theta}{2}\in\Big(0,\frac{\pi}{2}\Big],\qquad t:=bu=\frac{b\theta}{2}.
	\]
	Then $t\ge \frac{\pi}{20}$, and 
	\[
	R_b(\theta)=\Big|\frac{\sin t}{b\sin u}\Big|=\Big|\frac{\sin t}{t}\cdot \frac{u}{\sin u}\Big|.
	\]
	We consider two subcases:
	
\smallskip
\noindent\textbf{Subcase 2.1: $t\geq\pi$.}  For $u\in[0,\pi/2]$, the elementary inequality $\sin u\ge \frac{2}{\pi}u$ gives
	\[
	b\sin u \ge b\cdot \frac{2}{\pi}u=\frac{2}{\pi}t.
	\]
	Hence, 
	\[
	R_b(\theta)=\frac{|\sin t|}{b\sin u}\le \frac{1}{b\sin u}\le \frac{\pi}{2t}\le \frac12.
	\]
	
\smallskip
\noindent\textbf{Subcase 2.2: $\frac{\pi}{20}\le t\le \pi$.}
	Since $b\ge 2$, we have $u=t/b\le t/2\le \pi/2$.
	It is well known that the function $x\mapsto \frac{\sin x}{x}$ is decreasing on $(0,\pi]$.
	Therefore, from $0<u\le t/2$, we deduce 
	\[
	\frac{\sin u}{u}\ge \frac{\sin(t/2)}{t/2}
	\quad\Longrightarrow\quad
	\frac{u}{\sin u}\le \frac{t/2}{\sin(t/2)}.
	\]
	Consequently,
	\[
	R_b(\theta)
	=\Big|\frac{\sin t}{t}\cdot \frac{u}{\sin u}\Big|
	\le\Big| \frac{\sin t}{t}\cdot \frac{t/2}{\sin(t/2)}\Big|
	=\Big|\frac{\sin t}{2\sin(t/2)}\Big|
	=|\cos(t/2)|.
	\]
	Since $t/2\in[\pi/40,\pi/2]$, we have $|\cos(t/2)|=\cos(t/2)\le \cos(\pi/40)$.
	Thus
	\[
	R_b(\theta)\le \delta:=\cos(\pi/40)<1.
	\]
	Combining {\bf{Subcase 2.1}} and {\bf{Subcase 2.2}}, we obtain the uniform bound
	\[
	R_b(\theta)\le \delta=\cos(\pi/40)<1,
	\qquad\text{whenever } b|\theta|\ge \frac{\pi}{10},\ |\theta|\le \pi.
	\]
	Hence
	\[
	1-R(\theta)^2\ge 1-\delta^2=\sin^2(\pi/40).
	\]
	Recalling  \eqref{KAQ-sin}, we conclude  {\bf{Case 2}},
	\[
	|K_{A_Q}(x,y)|
	\le \min\{4,(b\theta)^2\}\,\bigl(1-R(\theta)^2\bigr)^{-1}
	\le 4\,\sin^{-2}(\pi/40).
	\]
	The proof is complete.
\end{proof}

\begin{proof} [Proof of Theorem \ref{NewCZ}]Let $f\in L_1^+(\mathcal{N})$ and $\lambda>0$.  Let $(q_{k+1,\ell})_{k\geq -1,1\leq \ell\leq\lceil\log_2 m_{k}\rceil}$ be the Cuculescu projections associated with $f$ and $\lambda$, and let   $q=\wedge_{k\geq-1}\wedge_{\ell=1}^{\lceil\log_2 m_{k}\rceil} q_{k+1,\ell}$.

\medskip
\noindent\textbf{Step 1: decomposition of $f$.} We begin with the elementary decomposition induced by the projection $q$:
	$$f=qfq+(1-q)fq+qf(1-q)+(1-q)f(1-q).$$
The term $qfq$ contributes to the good part, {while the remaining three terms will be expanded by
substituting the identity $ \eqref{1-q}$ into $(1-q)fq, qf(1-q)$ and $(1-q)f(1-q)$, } and we obtain
	\begin{align*}
		f&=qfq+\sum_{k\geq -1}\sum_{\ell=1}^{\lceil\log_2 m_{k}\rceil}\sum_{Q\in F_{k+1,\ell}}p_{Q}fq\chi_Q+\sum_{j\geq -1}\sum_{i=1}^{\lceil\log_2 m_{j}\rceil}\sum_{Q\in F_{j+1,i}}qfp_{Q}\chi_Q\\
		&\quad+\sum_{j,k\geq -1}\sum_{\ell=1}^{\lceil\log_2 m_{k}\rceil}\sum_{i=1}^{\lceil\log_2 m_{j}\rceil}\sum_{\substack{Q^1\in F_{k+1,\ell}\\Q^2\in F_{j+1,i}}}p_{Q^1}fp_{Q^2}\chi_{Q^1}\chi_{Q^2}.
	\end{align*}
Introducing $A_Q(f)$ as in \eqref{AQ}, we may further write
\begin{align*}
		f&=qfq+\sum_{k\geq -1}\sum_{\ell=1}^{\lceil\log_2 m_{k}\rceil}\sum_{Q\in F_{k+1,\ell}}p_{Q}A_Q(f)q\chi_Q+\sum_{j\geq -1}\sum_{i=1}^{\lceil\log_2 m_{j}\rceil}\sum_{Q\in F_{j+1,i}}qA_Q(f)p_{Q}\chi_Q\\
		&\quad+\sum_{j,k\geq -1}\sum_{\ell=1}^{\lceil\log_2 m_{k}\rceil}\sum_{i=1}^{\lceil\log_2 m_{j}\rceil}\sum_{\substack{Q^1\in F_{k+1,\ell}\\Q^2\in F_{j+1,i}}}p_{Q^1}A_{Q^1\wedge Q^2}(f)p_{Q^2}\chi_{Q^1}\chi_{Q^2}\\
		&\quad+\sum_{k\geq -1}\sum_{\ell=1}^{\lceil\log_2 m_{k}\rceil}\sum_{Q\in F_{k+1,\ell}}p_{Q}(f-A_Q(f))q\chi_Q+\sum_{j\geq -1}\sum_{i=1}^{\lceil\log_2 m_{j}\rceil}\sum_{Q\in F_{j+1,i}}q(f-A_Q(f))p_{Q}\chi_Q\\
		&\quad+\sum_{j,k\geq -1}\sum_{\ell=1}^{\lceil\log_2 m_{k}\rceil}\sum_{i=1}^{\lceil\log_2 m_{j}\rceil}\sum_{\substack{Q^1\in F_{k+1,\ell}\\Q^2\in F_{j+1,i}}}p_{Q^1}(f-A_{Q^1\wedge Q^2}(f))p_{Q^2}\chi_{Q^1}\chi_{Q^2}=:\sum_{n=1}^7f_n,
	\end{align*}
where we adopt the convention that $Q^1\wedge Q^2=Q^1$ if $\mathcal F_{k+1,\ell}\subseteq \mathcal F_{j+1,i}$, and $Q^1\wedge Q^2=Q^2$ otherwise. In the present setting,  the usual  commutation relations from the standard Cuculescu construction are unavailable, which gives rise to an additional off-diagonal good term  \(g_{\mathrm{off}}\). On the other hand, the monotonicity of the refined filtration implies that any two atoms are either nested or disjoint. This allows us to split the term \(f_4\) into one diagonal part and two off-diagonal parts:
	\begin{align*}
		f_4&=\sum_{j,k\geq -1}\sum_{\ell=1}^{\lceil\log_2 m_k\rceil}\sum_{i=1}^{\lceil\log_2 m_j\rceil}
		\sum_{\substack{Q^1\in\mathcal F_{k+1,\ell}\\ Q^2\in\mathcal F_{j+1,i}}}
		p_{Q^1}A_{Q^1\wedge Q^2}(f)p_{Q^2}\chi_{Q^1}\chi_{Q^2} \\
		&=\sum_{k\geq -1}\sum_{\ell=1}^{\lceil\log_2 m_k\rceil}\sum_{Q\in\mathcal F_{k+1,\ell}}
		p_QA_Q(f)p_Q\chi_Q \\
		&\quad
		+\sum_{k\geq -1}\sum_{\ell=1}^{\lceil\log_2 m_k\rceil}\sum_{Q\in\mathcal F_{k+1,\ell}}
		p_QA_Q(f)\Big(
		\sum_{i=\ell+1}^{\lceil\log_2 m_k\rceil}\sum_{Q^2\in\mathcal F_{k+1,i}} p_{Q^2}\chi_{Q^2}
		+\sum_{j>k}\sum_{i=1}^{\lceil\log_2 m_j\rceil}\sum_{Q^2\in\mathcal F_{j+1,i}} p_{Q^2}\chi_{Q^2}
		\Big)\chi_Q \\
		&\quad
		+\sum_{j\geq -1}\sum_{i=1}^{\lceil\log_2 m_j\rceil}\sum_{Q\in\mathcal F_{j+1,i}}
		\Big(
		\sum_{\ell=i+1}^{\lceil\log_2 m_j\rceil}\sum_{Q^1\in\mathcal F_{j+1,\ell}} p_{Q^1}\chi_{Q^1}
		+\sum_{k>j}\sum_{\ell=1}^{\lceil\log_2 m_k\rceil}\sum_{Q^1\in\mathcal F_{k+1,\ell}} p_{Q^1}\chi_{Q^1}
		\Big)A_Q(f)p_Q\chi_Q \\
		&=:f_{4,1}+f_{4,2}+f_{4,3},
	\end{align*}
The last two terms arise because the filtration  $(\mathcal{F}_{k+1,\ell})_{k\geq -1,\ 1\leq \ell\leq\lceil\log_2 m_k\rceil}$ is indexed by two parameters.  
We define the diagonal good part by
$$g_{\mathrm{d}}=qfq+f_{4,1}=qfq+\sum_{k\geq -1}\sum_{\ell=1}^{\lceil\log_2 m_{k}\rceil}\sum_{Q\in F_{k+1,\ell}}p_{Q}A_Q(f)p_Q\chi_Q.$$	
Moreover, by \eqref{qnlQ}-\eqref{1-q}, we have
\begin{equation}\label{f2+42}
	\begin{aligned}
		f_2+f_{4,2}
		&=
		\sum_{k\geq -1}\sum_{\ell=1}^{\lceil\log_2 m_k\rceil}\sum_{Q\in F_{k+1,\ell}}
		p_QA_Q(f)\Big(
		q
		+\sum_{i=\ell+1}^{\lceil\log_2 m_k\rceil}\sum_{Q^2\in F_{k+1,i}} p_{Q^2}\chi_{Q^2}\\
		&\hspace{4.8cm}
		+\sum_{j>k}\sum_{i=1}^{\lceil\log_2 m_j\rceil}\sum_{Q^2\in F_{j+1,i}} p_{Q^2}\chi_{Q^2}
		\Big)\chi_Q \\
		&=
		\sum_{k\geq -1}\sum_{\ell=1}^{\lceil\log_2 m_k\rceil}\sum_{Q\in F_{k+1,\ell}}
		p_QA_Q(f)q_Q\,\chi_Q,
	\end{aligned}
\end{equation}
and similarly,
$$f_3+f_{4,3}=\sum_{j\geq-1}\sum_{i=1}^{\lceil\log_2 m_{j}\rceil}\sum_{Q\in F_{j+1,i}}q_{Q}A_{Q}(f)p_Q\chi_{Q}.$$
Combining \eqref{f2+42} with the above identity, we obtain the off-diagonal good part
$$g_{\mathrm{off}}=f_2+f_{4,2}+f_3+f_{4,3}=\sum_{k\geq-1}\sum_{\ell=1}^{\lceil\log_2 m_{k}\rceil}\sum_{Q\in F_{k+1,\ell}}\Big(p_{Q}A_{Q}(f)q_Q\chi_{Q}+q_{Q}A_{Q}(f)p_Q\chi_{Q}\Big).$$
For the remaining terms $f_5,f_6,f_7$, we apply the same diagonal/off-diagonal splitting as above to \( f_7 \). This yields
$$b_{\mathrm{d}}=\sum_{k\geq -1}\sum_{\ell=1}^{\lceil\log_2 m_{k}\rceil}\sum_{Q\in F_{k+1,\ell}} p_Q(f-A_Q(f))p_Q\chi_Q$$
and
$$b_{\mathrm{off}}=\sum_{k\geq -1}\sum_{\ell=1}^{\lceil\log_2 m_{k}\rceil}\sum_{Q\in F_{k+1,\ell}}\Big(p_Q(f-A_Q(f))q_Q\chi_Q+q_Q(f-A_Q(f))p_Q\chi_Q\Big).$$
Collecting the preceding identities, we arrive at
$$f=g_{\mathrm{d}}+g_{\mathrm{off}}+b_{\mathrm{d}}+b_{\mathrm{off}},$$
which is exactly the decomposition in \eqref{dec}.

\medskip
\noindent\textbf{Step 2: proof of (i).}  From the definition of $\zeta$ in \eqref{zeta}, together with  \eqref{1-q}  and Lemma  \ref{Cuculescu}, we directly derive the first estimate
$$\lambda\varphi(\mathbf{1}-\zeta)\leq \lambda\sum_{Q\in F}3\varphi(p_Q\chi_Q)=3\lambda\sum_{k\geq -1}\sum_{\ell=1}^{\lceil\log_2 m_{k}\rceil}\varphi(p_{k+1,\ell})=3\lambda\varphi(1-q)\leq 3\|f\|_{L_1(\mathcal{N})}.$$

\medskip
\noindent\textbf{Step 3: proof of (ii).} Write  $A_Q(f)=\mathrm{Re}(A_Q(f))+i\mathrm{Im}(A_Q(f))$, where, for $x\in Q$,
\begin{equation}\label{real}
\mathrm{Re}(A_Q(f))(x)=\frac{1}{|Q|}\int_Qf(y) \mathrm{Re}(K_{A_Q}(x,y))d\mu(y)
\end{equation}
and
\begin{equation}\label{im}
	\mathrm{Im}(A_Q(f))(x)=\frac{1}{|Q|}\int_Qf(y) \mathrm{Im}(K_{A_Q}(x,y))d\mu(y).
\end{equation}

We first estimate \(\|g_{\mathrm d}\|_{L_1(\mathcal N)}\). Notice that $f$ is positive.   Combining \eqref{real}, \eqref{im} and Lemma \ref{KAQ}, we obtain
$$\|p_{Q}(\mathrm{Re}(A_Q(f))p_{Q}\chi_Q\|_{L_1(\mathcal{N})}+\|p_{Q}\mathrm{Im}(A_Q(f)))p_Q\chi_Q\|_{L_1(\mathcal{N})}\leq 2c\|p_{Q}f_{Q}p_Q\chi_Q\|_{L_1(\mathcal{N})}.$$
Consequently, by Lemma \ref{Cuculescu}(iv),  
\begin{align*}
\|g_{\mathrm{d}}\|_{L_1(\mathcal{N})}&=\|qfq+\sum_{k\geq -1}\sum_{\ell=1}^{\lceil\log_2 m_{k}\rceil}\sum_{Q\in F_{k+1,\ell}}p_{Q}(\mathrm{Re}(A_Q(f))+i\mathrm{Im}(A_Q(f)))p_Q\chi_Q\|_{L_1(\mathcal{N})}\\
&\leq \|qfq\|_{L_1(\mathcal{N})}+c\sum_{k\geq -1}\sum_{\ell=1}^{\lceil\log_2 m_{k}\rceil}\sum_{Q\in F_{k+1,\ell}}\|p_{Q}f_Qp_Q\chi_Q\|_{L_1(\mathcal{N})}\\
&=\varphi(qfq)+c\varphi((1-q)f)\leq c\|f\|_{L_1(\mathcal{N})}.
\end{align*}

For the estimate of $\|g_{\mathrm{d}}\|_{L_\infty(\mathcal{N})}$, the  projections \(q\) and $p_Q\chi_Q$ are disjoint,  hence 
$$\|g_{\mathrm{d}}\|_{L_\infty(\mathcal{N})}=\max\{\|qfq\|_{L_\infty(\mathcal{N})},\sup_{\substack{k\geq -1\\1\leq \ell\leq\lceil\log_2 m_{k}\rceil\\Q\in F_{k+1,\ell}}}\|p_{Q}A_Q(f)p_Q\chi_Q\|_{L_\infty(\mathcal{N})}\}.$$
Using \eqref{real}, \eqref{im} and Lemma \ref{KAQ} once more, we obtain
\begin{align*}
\sup_{\substack{k\geq -1\\1\leq \ell\leq\lceil\log_2 m_{k}\rceil\\Q\in F_{k+1,\ell}}}\|p_{Q}A_Q(f)p_Q\chi_Q\|_{L_\infty(\mathcal{N})}
	&\leq c\sup_{\substack{k\geq -1\\1\leq \ell\leq\lceil\log_2 m_{k}\rceil\\Q\in F_{k+1,\ell}}}\|p_{Q}f_Qp_Q\chi_Q\|_{L_\infty(\mathcal{N})}\\
	&\leq c \sup_{\substack{k\geq -1\\1\leq \ell\leq\lceil\log_2 m_{k}\rceil\\Q\in F_{k+1,\ell}}}\|p_{Q}f_{\widehat Q}p_Q\chi_Q\|_{L_\infty(\mathcal{N})}\\\
	&=c\sup_{\substack{k\geq -1\\1\leq \ell\leq\lceil\log_2 m_{k}\rceil\\Q\in F_{k+1,\ell}}}\|p_{Q}q_{\widehat Q}f_{\widehat Q}q_{\widehat Q}p_Q\chi_Q\|_{L_\infty(\mathcal{N})}\leq c\lambda,
\end{align*}
where the second and last inequalities are due to Lemma \ref{regular} and Lemma \ref{Cuculescu}(iii), respectively, the equality follows from the definition of $p_Q$. 
On the other hand, Lemma \ref{Cuculescu}(iii) also gives
\[
\|qfq\|_{L_\infty(\mathcal N)}\le \lambda.
\]
Combining the preceding estimates, we conclude that
\[
\|g_{\mathrm d}\|_{L_\infty(\mathcal N)}
\le
c\lambda.
\]

\medskip
\noindent\textbf{Step 4: proof of (iii).} To estimate  \eqref{goff-2-lambda f}, 
we decompose
\[
g_{\mathrm{off}}=g_{\mathrm{off},1}+g_{\mathrm{off},2},
\]
where
\[
g_{\mathrm{off},1}
:=\sum_{k\geq -1}\sum_{\ell=1}^{\lceil\log_2 m_k\rceil}
\sum_{Q\in\mathcal F_{k+1,\ell}} p_QA_Q(f)q_Q\chi_Q,
\]
and
\[
g_{\mathrm{off},2}
:=\sum_{k\geq -1}\sum_{\ell=1}^{\lceil\log_2 m_k\rceil}
\sum_{Q\in\mathcal F_{k+1,\ell}} q_QA_Q(f)p_Q\chi_Q.
\]
Then, we immediately obtain
$$\|g_{\mathrm{off}}\|_{L_2(\N)}^2\leq 2(\|g_{\mathrm{off},1}\|_{L_2(\N)}^2+\|g_{\mathrm{off},2}\|_{L_2(\N)}^2).$$
It is  enough to estimate $\|g_{\mathrm{off},1}\|_{L_2(\N)}^2$, since the term
$\|g_{\mathrm{off},2}\|_{L_2(\mathcal N)}^2$ can be handled exactly in the same way.
Since the  projections $p_Q\chi_Q$ are disjoint,  we have
$$(g_{\mathrm{off},1})^*g_{\mathrm{off},1}=\sum_{k\geq -1}\sum_{\ell=1}^{\lceil\log_2 m_{k}\rceil}\sum_{Q\in F_{k+1,\ell}}(p_{Q}A_{Q}(f)q_Q\chi_{Q})^*p_{Q}A_{Q}(f)q_Q\chi_{Q}.$$
Therefore,
\begin{equation}\label{goff-inf-1}
\begin{aligned}
	\|g_{\mathrm{off},1}\|_{L_2(\N)}^2&=\sum_{k\geq -1}\sum_{\ell=1}^{\lceil\log_2 m_{k}\rceil}\sum_{Q\in F_{k+1,\ell}}\|p_{Q}A_{Q}(f)q_Q\chi_{Q}\|_{L_2(\N)}^2\\
	&\leq \sum_{k\geq -1}\sum_{\ell=1}^{\lceil\log_2 m_{k}\rceil}\sum_{Q\in F_{k+1,\ell}}\|p_{Q}A_{Q}(f)q_Q\chi_{Q}\|_{L_\infty(\N)}\|p_{Q}A_{Q}(f)q_Q\chi_{Q}\|_{L_1(\N)}.
\end{aligned}  
\end{equation}
We next claim that
\begin{equation}\label{goff-infty}
	\sup_{Q\in F}\|p_{Q}A_Q(f)q_Q\|_{\infty}\leq c\lambda,
\end{equation}
and 
\begin{equation}\label{goff-1-sum}
\sum_{k\geq -1}\sum_{\ell=1}^{\lceil\log_2 m_{k}\rceil}\sum_{Q\in F_{k+1,\ell}} \|p_QA_Q(f)q_Q\chi_Q\|_{L_1(\mathcal{N})}\leq c\|f\|_{L_1(\mathcal{N})}.
\end{equation}
Assuming these claims  hold, substituting \eqref{goff-infty} and \eqref{goff-1-sum} into \eqref{goff-inf-1} directly gives
$$	\|g_{\mathrm{off},1}\|_{L_2(N)}^2\leq c\lambda \|f\|_{L_1(\mathcal{N})}.$$

We now prove the claim \eqref{goff-infty}. Fix an arbitrary atom $Q\in F_{k+1,\ell}$. By Lemma \ref{AcB}, for every $x\in Q$, we have
\begin{equation}\label{B1-u-B2}
	p_{Q}A_Q(f)(x)q_Q=\frac{1}{|Q|}\int_Qp_{Q}f(y)K_{A_Q}(x,y)q_Qd\mu(y)=\frac{1}{|Q|}B_1(x)^{\frac{1}{2}}\cdot u_{f,Q}\cdot B_2^{\frac{1}{2}},
\end{equation}
where 
$$B_1(x)=\int_Qp_Qf(y)p_Q|K_{A_Q}(x,y)|^2d\mu(y),\quad B_2=\int_Qq_Qf(y)q_Qd\mu(y)$$
and
$$\|u_{f,Q}\|_{L_\infty(\mathcal{N})}\leq1.$$
Consequently,
$$\|p_{Q}A_Q(f)q_Q\chi_Q\|_{L_\infty(\mathcal{N})}\leq \frac{1}{|Q|}\|B_1(x)^{\frac{1}{2}}\|_{L_\infty(\mathcal{M})}\|B_2^{\frac{1}{2}}\|_{L_\infty(\mathcal{M})}.$$
By Lemma \ref{KAQ}, Lemma \ref{regular} and Lemma \ref{Cuculescu} (iii),  we obtain
\begin{align*}
	\|B_1(x)^{\frac{1}{2}}\|_{L_\infty(\mathcal{M})}&\leq c\||Q|p_Qf_Qp_Q\|_{L_\infty(\mathcal{M})}^{\frac{1}{2}}\leq c|Q|^{\frac{1}{2}}\|p_Qf_{\widehat{Q}}p_Q\|_{L_\infty(\mathcal{M})}^{\frac{1}{2}}\\
	&\leq c|Q|^{\frac{1}{2}}\|p_Qq_{\widehat{Q}}f_{\widehat{Q}}q_{\widehat{Q}}p_Q\|_{L_\infty(\mathcal{M})}^{\frac{1}{2}}\leq c|Q|^{\frac{1}{2}}\cdot\lambda^{\frac{1}{2}}.
\end{align*}
On the other hand, Lemma~\ref{Cuculescu}(iii) gives
\begin{equation}\label{B2-infty}
	\|B_2^{\frac{1}{2}}\|_{L_\infty(\mathcal{M})}=\||Q|q_Qf_Qq_Q\|_{L_\infty(\mathcal{M})}^{\frac{1}{2}}\leq |Q|^{\frac{1}{2}}\lambda^{\frac{1}{2}}.
\end{equation}
Combining the above estimates, we complete the proof of the claim \eqref{goff-infty}.

We next prove \eqref{goff-1-sum}. By \eqref{B1-u-B2}, we have
\begin{align*}
	\|p_QA_Q(f)q_Q\chi_Q\|_{L_1(\mathcal{N})}&=\int_Q\|\frac{1}{|Q|}B_1(x)^{\frac 1 2}\cdot u_{f,Q}\cdot B_2^{\frac 1 2}\|_{L_1(\mathcal{M})}d\mu(x)\\
	&\leq\int_Q\frac{1}{|Q|}\Big\|B_{1}(x)^{\frac 1 2}\Big\|_{L_1(\mathcal{M})}\cdot\|B_2^{\frac 1 2}\|_{L_\infty(\mathcal{M})}d\mu(x).
\end{align*}
Using the H\"{o}lder inequality and Lemma \ref{KAQ}, we get
$$
	\|B_1(x)^{\frac{1}{2}}\|_{L_1(\mathcal{M})}=\Big\|p_QB_1(x)\Big\|_{L_{\frac{1}{2}}(\mathcal{M})}^{\frac{1}{2}}\leq c \|p_Q\|_{L_1(\mathcal{M})}^{\frac{1}{2}}\Big\|\int_Qp_Qf(y)p_Qd\mu(y)\Big\|_{L_1(\mathcal{M})}^{\frac{1}{2}}.
$$
Then by the above estimate and \eqref{B2-infty}, we obtain
\begin{equation}\label{paq-L1N}
	\begin{aligned}
		\|p_QA_Q(f)q_Q\chi_Q\|_{L_1(\mathcal{N})}&\leq c |Q|^\frac{1}{2}\lambda^\frac{1}{2}\|p_Q\|_{L_1(\mathcal{M})}^{\frac{1}{2}}\Big\|\int_Qp_Qf(y)p_Qd\mu(y)\Big\|_{L_1(\mathcal{M})}^{\frac{1}{2}}\\
		&=c\lambda^\frac{1}{2} \varphi(p_Q\chi_Q)^\frac{1}{2}\cdot\varphi(p_Qfp_Q\chi_Q)^\frac{1}{2}.
	\end{aligned}
\end{equation}
Therefore, by the Cauchy-Schwarz inequality and Lemma \ref{Cuculescu} (iv), 
\begin{align*}
&\sum_{k\geq -1}\sum_{\ell=1}^{\lceil\log_2 m_{k}\rceil}\sum_{Q\in F_{k+1,\ell}} \|p_QA_Q(f)q_Q\chi_Q\|_{L_1(\mathcal{N})}\\
	&\leq c\sum_{k\geq -1}\sum_{\ell=1}^{\lceil\log_2 m_{k}\rceil}\sum_{Q\in F_{k+1,\ell}} \lambda^\frac{1}{2} \varphi(p_Q\chi_Q)^\frac{1}{2}\cdot\varphi(p_Qfp_Q\chi_Q)^\frac{1}{2}\\
	&\leq \Big(\sum_{k\geq -1}\sum_{\ell=1}^{\lceil\log_2 m_{k}\rceil}\sum_{Q\in F_{k+1,\ell}} \lambda \varphi(p_Q\chi_Q)\Big)^\frac{1}{2}\cdot\Big(\sum_{k\geq -1}\sum_{\ell=1}^{\lceil\log_2 m_{k}\rceil}\sum_{Q\in F_{k+1,\ell}}\varphi(p_Qfp_Q\chi_Q)\Big)^\frac{1}{2}\\
	&\leq c(\lambda \varphi(1-q))^{\frac 12}(\varphi((1-q)f))^{\frac{1}{2}}\leq c\leq\|f\|_{L_1(\mathcal{N})}.
\end{align*}
Consequently, \eqref{goff-1-sum} follows.

\medskip
\noindent\textbf{Step 5: proof of (iv).} We first verify the  cancellation conditions \eqref{bd0}. If $Q\in F_{k+1}'$ or $\alpha_k=0$, then $A_Q(f)=f_Q$, and thus
$$
\int_Qb_{\mathrm{d}}^{k+1,\ell,Q}(x)d\mu(x)=\int_Qp_Q(f(x)-f_Q)p_Q\chi_Q(x)d\mu(x)=0$$
and
$$
\int_Qb_{\mathrm{d}}^{k+1,\ell,Q}r_k^{\alpha_k}(x)d\mu(x)=\int_Qp_Q(f(x)-f_Q)p_Q\chi_Q(x)r_k^{\alpha_k}(x)d\mu(x)=0,$$
where the last equality holds since $r_k^{\alpha_k}$ is constant whenever $Q\in F_{k+1}'$ or $\alpha_k=0$.
Otherwise, by  \eqref{AQ1} and \eqref{AQ}, we have
$$A_Q(f)(x)=\frac{f_Q-(fr_k^{\alpha_k})_Q\cdot(r_k^{-\alpha_k})_Q+\big[(fr_k^{\alpha_k})_Q-f_Q\cdot(r_k^{\alpha_k})_Q\big]\cdot r_k^{-\alpha_k}(x)}{1-\big|\frac{1}{|Q|} \int_{Q} r_k^{\alpha_{k}} d \mu\big|^{2}}.$$
Then,
\begin{equation*}
\begin{aligned}
&\int_Q(f(x)-A_Q(f)(x))d\mu(x)\\
&=\int_Qf(x)d\mu(x)-\frac{|Q|f_Q-(fr_k^{\alpha_k})_Q\cdot|Q|\cdot(r_k^{-\alpha_k})_Q+\big[(fr_k^{\alpha_k})_Q-f_Q\cdot(r_k^{\alpha_k})_Q\big]\cdot \int_Qr_k^{-\alpha_k}(x)d\mu(x)}{1-\big|\frac{1}{|Q|} \int_{Q} r_k^{\alpha_{k}} d \mu\big|^{2}}\\
&=\int_Qf(x)d\mu(x)-\frac{|Q|f_Q-f_Q(r_k^{\alpha_k})_Q\cdot \int_Qr_k^{-\alpha_k}(x)d\mu(x)}{1-\big|\frac{1}{|Q|} \int_{Q} r_k^{\alpha_{k}} d \mu\big|^{2}}\\
&=\int_Qf(x)d\mu(x)-\frac{\int_Qfd\mu\Big(1-\frac{1}{|Q|}\int_Qr_k^{\alpha_k}d\mu\cdot \frac{1}{|Q|}\int_Qr_k^{-\alpha_k}d\mu\Big)}{1-\big|\frac{1}{|Q|} \int_{Q} r_k^{\alpha_{k}} d \mu\big|^{2}}=0.
\end{aligned}
\end{equation*}
This immediately yields
$$
\int_Qb_{\mathrm{d}}^{k+1,\ell,Q}(x)d\mu(x)=\int_Qp_Q(f(x)-A_Q(f)(x))p_Q\chi_Q(x)d\mu(x)=0.$$
Similarly, it follows that
$$
\int_Qb_{\mathrm{d}}^{k+1,\ell,Q}r_k^{\alpha_k}(x)d\mu(x)=\int_Qp_Q(f(x)-A_Q(f)(x))p_Q\chi_Q(x)r_k^{\alpha_k}(x)d\mu(x)=0.$$
We finish the proof of \eqref{bd0}. Moreover, we deduce that
\begin{align*}
	&\sum_{k\geq -1}\sum_{\ell=1}^{\lceil\log_2 m_{k}\rceil}\sum_{Q\in F_{k+1,\ell}}\|b_{\mathrm{d}}^{k+1,\ell,Q}\|_{L_1(\mathcal{N})}\\
	&=\sum_{k\geq -1}\sum_{\ell=1}^{\lceil\log_2 m_{k}\rceil}\sum_{Q\in F_{k+1,\ell}}\Big\|p_Qfp_Q\chi_Q-p_{Q}\mathrm{Re}(A_Q(f))p_Q\chi_Q-ip_Q\mathrm{Im}(A_Q(f))p_Q\chi_Q\Big\|_{L_1(\mathcal{N})}\\
	&\leq \sum_{k\geq -1}\sum_{\ell=1}^{\lceil\log_2 m_{k}\rceil}\sum_{Q\in F_{k+1,\ell}}\Big(\|p_Qfp_Q\chi_Q\|\|_{L_1(\mathcal{N})}+c\|p_{Q}f_Qp_Q\chi_Q\|_{L_1(\mathcal{N})}\Big)=c\varphi((1-q)f)\leq c\|f\|_{L_1(\mathcal{N})},
\end{align*}
 where we have used Lemma \ref{KAQ} (resp. Lemma \ref{Cuculescu} (iv)) in the first (resp. last) inequality.

\medskip
\noindent\textbf{Step 6: proof of (v).} Following the same argument as in the proof of cancellation conditions stated in (iv), for every $Q\in F_{k+1,\ell}$, we obtain \eqref{boff0}.
This completes the proof.
\end{proof}

\section{Proof of Theorem  \ref{weak-main}}\label{sec4}
In this section, we prove Theorem~\ref{weak-main} by passing to a twisted version of the Vilenkin partial sums. The advantage of this renormalization is that the corresponding kernels admit a decomposition into local averaging terms and Hilbert-transform-type terms, which fits naturally with the Calder\'on-Zygmund decomposition established in Section~\ref{3.2}. We first prove uniform weak type \((1,1)\) and strong type \((p,p)\) estimates for the modified operators \(\widetilde S_n\), the desired bounds for the original partial sums \(S_n\) will then follow by a simple argument.

For \(n\in\mathbb N\), define the \(n\)-th modified partial sum by
$$\widetilde{S}_n(f)(x)=\int_0^1f(y) \widetilde{D}_n(x\dot{-}y)d\mu(y),$$
where
$$ \widetilde{D}_n(t)= \overline{\phi_n(t)}{D_n}(t).$$
Moreover, it follows from \cite[p. 348]{Go1973} that
\begin{equation}\label{SS}
 \widetilde{S}_n(f) =\overline{\phi_n} S_n(\phi_nf).
\end{equation}
We first establish the following uniform boundedness for $\widetilde{S}_n$, from which
Theorem \ref{weak-main} will follow immediately.
\begin{proposition}\label{hatS} There exists a  universal constants $c>0$ such that
	$$\sup_{n\geq1}\|\widetilde{S}_nf\|_{L_{1,\infty}(\mathcal {N})}\leq c\|f\|_{L_1(\mathcal{N})}, \quad f\in L_1(\mathcal{N})$$
and, for every $1<p<\infty$,
$$\sup_{n\geq1}\|\widetilde{S}_n(f)\|_{L_p(\mathcal{N})}\leq c\frac{p}{p-1}\|f\|_{L_p(\mathcal{N})},\quad f \in L_p(\mathcal{N}).$$
\end{proposition}

We now deduce Theorem~\ref{weak-main} from Proposition~\ref{hatS}.
\begin{proof} [Proof of Theorem \ref{weak-main}] 
For every  $n\geq1$, by Lemma \ref{propetyVilenkin} (i) and \eqref{SS}, we have $$|S_n(f)|=|\overline{\phi_n}S_n(f)|=|\overline{\phi_n}S_n((\phi_n\overline{\phi_n})f)|=|\widetilde{S}_n(\overline{\phi_n}f)|.$$
Applying Proposition \ref{hatS}, we obtain
\begin{align*}
\|S_n(f)\|_{L_{1,\infty}(\mathcal{N})}=\|\widetilde{S}_n(\overline{\phi_n}f)\|_{L_{1,\infty}(\mathcal{N})} \leq c\|\overline{\phi_n}f\|_{L_1(\mathcal{N})}=c\|f\|_{L_1(\mathcal{N})}.
\end{align*}
Similarly, the strong type $(p,p)$ estimate follows from the second inequality in
Proposition \ref{hatS}. This completes the proof.
\end{proof}

We now turn to the proof of Proposition \ref{hatS}.  We begin with recalling a decomposition
of $\widetilde{S}_n$ ( see \cite[Page 312, 313]{Yo1976}): given $x\in[0,1)$, suppose that $n=\sum_{k=0}^\infty\alpha_kM_k$ with $0\leq\alpha_k<m_k$. Then
\begin{equation}\label{Sn-S-alpha-k}
	\widetilde{S}_n(f)(x)=\sum_{k=0}^\infty\widetilde{S}_{\alpha_kM_k}(f)(x),
\end{equation}
and for each $k\ge 0$,
\begin{equation}\label{S-alpha-H}
	\begin{aligned} 
		\widetilde{S}_{\alpha_k M_k} f(x) &= \frac{\alpha_k}{|I_k(x)|} \int_{I_k(x) \cap \{x_k = t_k\}} f(t) \, d\mu(t)\\
		&\quad +  \frac{r_k^{-\alpha_k}(x)}{2|I_k(x)|} \int_{I_k(x) \cap \{x_k \neq t_k\}} f(t)\, r_k^{\alpha_k}(t) \, d\mu(t)  \\
		&\quad -  \frac{1}{2|I_k(x)|} \int_{I_k(x) \cap \{x_k \neq t_k\}} f(t) \, d\mu(t)  \\
		&\quad + i \, r_k^{-\alpha_k}(x) H_k(fr_k^{\alpha_k})(x) - i H_k (f)(x)\\
		&=:\sum_{i=1}^5\widetilde{S}_{\alpha_kM_k}^i(f)(x),
	\end{aligned}
\end{equation}
where $I_k(x)$ denotes the unique atom of $\mathcal{F}_k$ containing $x$, and 
\begin{equation}\label{Hk}
H_k(f)(x)=\frac{1}{2|I_k(x)|}\int_{I_k(x)\cap\{x_k\neq t_k\}}f(t)\cot(\frac{\pi(x_k-t_k)}{m_k})d\mu(t).
\end{equation}

The next elementary kernel estimate will be used repeatedly to handle the diagonal and off-diagonal bad terms.
\begin{lemma}\label{cot} 	Let $k\ge0$, $1\le \ell\le \lceil\log_2 m_k\rceil$, let $I\in F_k'$, and let
	$Q\in F_{k+1,\ell}$ be such that $Q\subset I$. Then  for every $t,t^Q\in Q$ with $t=\sum_{k=0}^\infty \frac{t_k}{M_{k+1}}$ and $t^Q=\sum_{k=0}^\infty \frac{t_k^Q}{M_{k+1}}$, and $J=[\frac{j}{M_{k+1}},\frac{j+1}{M_{k+1}})\subset I\cap(3Q)^c$, there exists a universal  constant $c>0$  such that  the following hold.
	\begin{enumerate}[\rm(i)]
		\item
		\[
		\int_J |\Delta_{Q}(x,t,t^Q)|\,d\mu(x)
		\le \frac{c}{M_k}\,\frac{|t_k-t_k^Q|}{\min\{|j-t_k|,|j-t_k^Q|\}(\min\{|j-t_k|,|j-t_k^Q|\}+|t_k-t_k^Q|)},
		\]
	where
	\begin{equation}\label{DQ}
	\Delta_{Q}(x,t,t^Q)
	:=
	\cot\!\big(\frac{\pi(x_k-t_k)}{m_k}\big)
	-
	\cot\!\big(\frac{\pi(x_k-t_k^Q)}{m_k}\big);
	\end{equation}
moreover,
		\[
		\frac1{|I|}\int_{I\cap(3Q)^c} |\Delta_{Q}(x,t,t^Q)|\,d\mu(x)\le c.
		\]
		
		\item
		\[
		\bigl(\int_J |\Delta_{Q}(x,t,t^Q)|^2\,d\mu(x)\bigr)^{1/2}
		\le
		\frac{c\,m_k^{1/2}}{M_k^{1/2}}\,\frac{|t_k-t_k^Q|}{\min\{|j-t_k|,|j-t_k^Q|\}(\min\{|j-t_k|,|j-t_k^Q|\}+|t_k-t_k^Q|)}.
		\]
	\end{enumerate}
\end{lemma}
\begin{proof}  Fix  $t,t^Q\in Q$ and  a level-$(k+1)$ atom $J=\bigl[\frac{j}{M_{k+1}},\frac{j+1}{M_{k+1}}\bigr)\subset I\cap(3Q)^c$.  If $t_k=t_k^Q$, then $\Delta_Q(x,t,t^Q)=0$, and all three estimates are trivial. Therefore, we may assume that $t_k\neq t_k^Q$.
Interchanging $t$ and $t^Q$ if necessary,  we may suppose that
\[|j-t_k|\le |j-t_k^Q|.\]
Set
\[a:=|j-t_k|,\qquad d:=|t_k-t_k^Q|\geq1.\]
Since $J\subset I\cap(3Q)^c$, we have
\[	|j-t_k^Q|=a+d\quad \mbox{and}\quad 1\le d\le a\le m_k-1.	\]
Using the identity
\[	\cot u-\cot v=\frac{\sin(v-u)}{\sin u\,\sin v},\]
we obtain
	\[	|\Delta_Q(x,t,t^Q)|
		=
		\frac{\bigl|\sin\bigl(\frac{\pi(t_k-t_k^Q)}{m_k}\bigr)\bigr|}
		{\bigl|\sin\bigl(\frac{\pi(x_k-t_k)}{m_k}\bigr)\bigr|
			\bigl|\sin\bigl(\frac{\pi(x_k-t_k^Q)}{m_k}\bigr)\bigr|}.
		\]
		Since $1\le d\le a\le a+d\le m_k-1$, the elementary estimate
		\[
		\sin\Bigl(\frac{\pi r}{m_k}\Bigr)\approx \frac{r}{m_k},
		\qquad 1\le r\le m_k-1,
		\]
		yields
		\[
		|\Delta_Q(x,t,t^Q)|
		\le
		c\,\frac{\frac{d}{m_k}}{\frac{a}{m_k}\frac{a+d}{m_k}}
		=
		c\,\frac{m_k\,d}{a(a+d)},\quad \mbox{for all}\ x\in J.
		\]
		Since $|J|=M_{k+1}^{-1}=(m_kM_k)^{-1}$, it follows that
		\[
		\int_J |\Delta_Q(x,t,t^Q)|\,d\mu(x)
		\le
		|J|\,\sup_{x\in J}|\Delta_Q(x,t,t^Q)|
		\le
		\frac{c}{M_k}\,\frac{d}{a(a+d)},
		\]
		which proves the first part of \rm(i), since
		\[
		a=\min\{|j-t_k|,|j-t_k^Q|\}.
		\]
		
		Similarly,
		\[
		\Big(\int_J |\Delta_Q(x,t,t^Q)|^2\,d\mu(x)\Big)^{1/2}
		\le
		|J|^{1/2}\,\sup_{x\in J}|\Delta_Q(x,t,t^Q)|
		\le
		\frac{c\,m_k^{1/2}}{M_k^{1/2}}\,
		\frac{d}{a(a+d)},
		\]
		which proves \rm(ii).
		
		It remains to prove the second assertion in \rm(i). Decompose $I\cap(3Q)^c$ into level-$(k+1)$ atoms and sum the preceding estimate.  We obtain
	\begin{align*}
	\frac1{|I|}\int_{I\cap(3Q)^c} |\Delta_{Q}(x,t,t^Q)|\,d\mu(x)
	&=M_k \sum_{\substack{I\in F_k', Q\in F_{k+1,\ell}, Q\subset I\\J\subset I\cap(3Q)^c}}\int_J |\Delta_{Q}(x,t,t^Q)|\,d\mu(x)\\
	&\le M_k\sum_{a= d}^{m_k}\frac{c}{M_k}\frac{d}{a(a+d)}\\
	&=c\int_{\mathrm{d}}^{m_k}\frac{d}{x(x+d)}dx=c\log(\frac{2m_k}{m_k+d})<c\log 2,
	\end{align*}
where the first inequality follows from the fact that there are no more than $m_k-3|Q|M_{k+1}$ such terms $J$, note that $m_k-3|Q|M_{k+1}\leq m_k-3d\leq m_k-d$.
	This completes the proof.
	\end{proof}
To prove Proposition \ref{hatS}, we establish the following two lemmas.

\begin{lemma}\label{bd-w11} There exists an  absolute constant $c>0$ such that
$$\varphi(\chi_{(\lambda,\infty)}(|\zeta\widetilde{S}_n(b_{\mathrm{d}})\zeta|))\leq \frac{c\|f\|_{L_1(\mathcal{N})}}{\lambda},\quad \forall n\geq1,$$
{where \(\zeta \) is the projection defined in \eqref{zeta}.}
\end{lemma}
\begin{proof} Consider an arbitrary positive integer $n$ with representation $n=\sum_{k=0}^\infty\alpha_kM_k$, $0\leq\alpha_k<m_k$. Without loss of generality, assume that $f\in L_1^+(\mathcal{N})$.  By \eqref{Sn-S-alpha-k} and \eqref{S-alpha-H},  we may write
	\begin{equation}\label{bd-S5}
		\zeta\widetilde{S}_n(b_{\mathrm{d}})\zeta=\sum_{k=0}^\infty\zeta\widetilde{S}_{\alpha_kM_k}(b_{\mathrm{d}})\zeta=\sum_{k=0}^\infty\sum_{i=1}^5\zeta\widetilde{S}_{\alpha_kM_k}^i(b_{\mathrm{d}})\zeta.
	\end{equation}
	Recall that
	\begin{equation}\label{bd-k+1-l-w11}
b_{\mathrm{d}}=\sum_{j\ge -1}\ \sum_{\ell=1}^{\lceil\log_2 m_j\rceil}\ \sum_{Q\in F_{j+1,\ell}} b_{\mathrm{d}}^{j+1,\ell,Q},
\qquad
b_{\mathrm{d}}^{j+1,\ell,Q}:=p_Q\,(f-A_Q(f))\,p_Q\,\chi_Q.
	\end{equation}
	Accordingly, for every $k\geq0$, $1\leq i\leq5$ and $x\in [0,1)$,
	\begin{equation}\label{bd-Si}
	\zeta(x)\widetilde{S}_{\alpha_kM_k}^i(b_{\mathrm{d}})(x)\zeta(x)=\sum_{j\geq -1}\sum_{\ell=1}^{\lceil\log_2 m_{j}\rceil}\sum_{Q\in F_{j+1,\ell}}\zeta(x) \widetilde{S}_{\alpha_kM_k}^i(b_{\mathrm{d}}^{j+1,\ell,Q})(x)\zeta(x).
	\end{equation} 
	We distinguish cases according to whether $x\in 3Q$ or $x\in (3Q)^c$, and further subdivide based on the sizes of  $j$ and $k$. We claim that the only terms on the right-hand side of  \eqref{bd-Si} that  contribute are those corresponding to the case
	\begin{equation}\label{claim}
	j=k,\ Q\in F_{k+1,\ell},\ x\in (3Q)^c,\ \mbox{and} \ i\in\{4,5\}.
	\end{equation}

\medskip
\noindent\textbf{Case 1: \(x\in 3Q\).} By the  construction of $\zeta$  in \eqref{zeta}, we have $\zeta(x)\leq p_Q^{\perp}$.  Since $b_{\mathrm d}^{j+1,\ell,Q}=p_Q b_{\mathrm d}^{j+1,\ell,Q} p_Q$, it follows that
	$$\zeta(x)\widetilde{S}_{\alpha_kM_k}^i(b_{\mathrm{d}}^{j+1,\ell,Q})(x)\zeta(x)=0,\quad 1\leq i\leq5.$$

	\medskip
	\noindent\textbf{Case 2: \(x\in (3Q)^c\).}
	We subdivide this case according to the sizes  of \(j\) and \(k\).
	
	\smallskip
	\noindent\textbf{Subcase 2.1: \(j<k\).} Since $Q\in F_{j+1,\ell}$, we have $|Q|\ge M_{j+1}^{-1}\ge M_k^{-1}=|I_k(x)|$. Hence, if $I_k(x)\cap Q\neq\emptyset$, then the interval $I_k(x)$ is contained in $3Q$, contradicting $x\in(3Q)^c$. Therefore
	\[
	I_k(x)\cap Q=\emptyset,
	\]
	and recalling  \eqref{S-alpha-H}, we have
	\[\zeta(x)\widetilde S_{\alpha_kM_k}^i\!\bigl(b_{\mathrm d}^{j+1,\ell,Q}\bigr)(x)\zeta(x)=0,
	\qquad 1\le i\le 5.\]
Indeed, in the subsequent proof, we only need to consider the case 
\begin{equation}\label{QIk}
Q\subset I_k(x),
\end{equation}
otherwise,  \eqref{S-alpha-H}  implies
	\[\zeta(x)\widetilde S_{\alpha_kM_k}^i\!\bigl(b_{\mathrm d}^{j+1,\ell,Q}\bigr)(x)\zeta(x)=0,
\qquad 1\le i\le 5.\]
 \smallskip
\noindent\textbf{Subcase 2.2: \(j\ge k+1\).} 
By $Q\in F_{j+1,\ell}$, we have $|Q|=\frac{b}{M_{j+1}}<\frac{1}{M_{k+1}}=|I_{k+1}(x)|$ for some $1\leq b< m_{j+1}$, $b\in \mathbb{N}^+$. There exists a unique atom $\widetilde{Q}\in F_{k+1}$ containing $Q$, so that either  $I_{k+1}(x)\cap \widetilde{Q}=\emptyset$ or $I_{k+1}(x)= \widetilde{Q}$,  that is, either $$I_{k+1}(x)\cap Q=\emptyset\quad \mbox{or} \quad Q\subset I_{k+1}(x).$$

For $i=1$,  by \eqref{S-alpha-H},  we have
\begin{equation}\label{i=1-j=k+1}
\zeta(x)\widetilde{S}_{\alpha_kM_k}^1(b_{\mathrm{d}}^{j+1,\ell,Q})(x)\zeta(x)=\zeta(x)\frac{\alpha_k}{|I_k(x)|} \int_{Q\cap  I_{k+1}(x)}b_{\mathrm{d}}^{j+1,\ell,Q}(t)d\mu(t)\zeta(x).
\end{equation}
In the first case, this term is obviously zero since the integration domain is empty; 
in the latter  case, \eqref{i=1-j=k+1} and Theorem~\ref{NewCZ} (iv) imply
$$\zeta(x)\widetilde{S}_{\alpha_kM_k}^1(b_{\mathrm{d}}^{j+1,\ell,Q})(x)\zeta(x)=\zeta(x)\frac{\alpha_k}{|I_k(x)|} \int_{Q}b_{\mathrm{d}}^{j+1,\ell,Q}(t)d\mu(t)\zeta(x)=0.$$

For $i=2$,   by \eqref{S-alpha-H},
\begin{equation}\label{i=2-j=k+1}
\zeta(x)\widetilde{S}_{\alpha_kM_k}^2(b_{\mathrm{d}}^{j+1,\ell,Q})(x)\zeta(x)=\zeta(x)\frac{r^{-\alpha_k}(x)}{2|I_k(x)|} \int_{Q\cap I_k(x) \cap \{x_k \neq t_k\}} b_{\mathrm{d}}^{j+1,\ell,Q}(t) r_k^{\alpha_k}(t)d\mu(t)\zeta(x).
\end{equation} 
It follows from \eqref{QIk} that the integration domain for this term is
\begin{equation}\label{Q-d}
	Q\cap I_k(x) \cap \{x_k \neq t_k\}=Q\setminus (Q\cap I_{k+1}(x)).
\end{equation}
The integration domain is empty when $Q\subset I_{k+1}(x)$. If $I_{k+1}(x)\cap Q=\emptyset$, the domain equals $Q$, and since $r_k^{\alpha_k}$ is constant on $Q$. Then \eqref{i=2-j=k+1} and Theorem \ref{NewCZ}\,(iv) imply
$$\zeta(x)\widetilde{S}_{\alpha_kM_k}^2(b_{\mathrm{d}}^{j+1,\ell,Q})(x)\zeta(x)=\zeta(x)\frac{cr^{-\alpha_k}(x)}{2|I_k(x)|} \int_{Q} b_{\mathrm{d}}^{j+1,\ell,Q}(t)d\mu(t)\zeta(x)=0.$$

The same argument applies to $i=3$. Since the underlying integration domain coincides with that for $i=2$.

We next consider $i=4$ and $i=5$. Both
$r_k^{\alpha_k}(t)$ and $\cot\!\bigl(\frac{\pi(x_k-t_k)}{m_k}\bigr)$ are constant on $t\in Q\in F_{j+1,\ell}$. Arguing exactly as above and using the cancellation $\int_Q b_{\mathrm d}^{j+1,\ell,Q}(t)\,d\mu(t)=0$ from Theorem  \ref{NewCZ} (iv),
we obtain
\[
\zeta(x)H_k\!\bigl(b_{\mathrm d}^{j+1,\ell,Q}r_k^{\alpha_k}\bigr)(x)\zeta(x)=0,
\qquad
\zeta(x)H_k\!\bigl(b_{\mathrm d}^{j+1,\ell,Q}\bigr)(x)\zeta(x)=0.
\]
Hence,
$$ \zeta(x)\widetilde{S}_{\alpha_kM_k}^i(b_{\mathrm{d}}^{j+1,\ell,Q})(x)\zeta(x),\quad i\in\{4,5\}.$$

	\smallskip
\noindent\textbf{Subcase 2.3: \(j=k\).} Since $Q\in F_{k+1,\ell}$, and the same geometric argument gives \begin{equation}\label{Q-I-k+1}
I_{k+1}(x)\cap Q=\emptyset.
\end{equation}

For $i=1$,  by \eqref{i=1-j=k+1}, the integration domain is empty, thus we have
$$\zeta(x)\widetilde{S}_{\alpha_kM_k}^1(b_{\mathrm{d}}^{j+1,\ell,Q})(x)\zeta(x)=0.$$
	
 For $i=2$, by \eqref{QIk}, \eqref{Q-d} and \eqref{Q-I-k+1}, we have
$$Q\cap I_k(x) \cap \{x_k \neq t_k\}=Q.$$
 By \eqref{i=2-j=k+1} and the cancellation property  $\int_Q b_{\mathrm d}^{k+1,\ell,Q}(x)\,r_k^{\alpha_k}(x)\,d\mu(x)=0$  from Theorem~\ref{NewCZ}(iv), we obtain 
$$ \zeta(x)\widetilde{S}_{\alpha_kM_k}^2(b_{\mathrm{d}}^{j+1,\ell,Q})(x)\zeta(x)
 =\zeta(x)\frac{r^{-\alpha_k}(x)}{2|I_k(x)|} \int_{Q} b_{\mathrm{d}}^{j+1,\ell,Q}(t) r_k^{\alpha_k}(t)d\mu(t)\zeta(x)=0.$$

 The same argument applies to \(i=3\) since the underlying integration domain coincides with that for $i=2$, now using the cancellation $ \int_Q b_{\mathrm d}^{k+1,\ell,Q}(t)\,d\mu(t)=0$  from Theorem~\ref{NewCZ} (iv). We now conclude the proof for the claim \eqref{claim}. Thus, combining  \eqref{bd-S5},  \eqref{bd-Si} and the claim \eqref{claim}, we arrive at
\[\zeta \widetilde S_n(b_{\mathrm d})\zeta=A-B,\]
where
$$A(x)=\zeta(x)\sum_{k\geq0}\sum_{\ell=1}^{\lceil\log_2 m_k\rceil}\sum_{Q\in F_{k+1,\ell}}i r_k^{-\alpha_k}(x) H_k(b_{\mathrm{d}}^{k+1,\ell,Q}r_k^{\alpha_k})(x) \zeta(x)$$
and
	$$B(x)=\zeta(x)\sum_{k\geq0}\sum_{\ell=1}^{\lceil\log_2 m_k\rceil}\sum_{Q\in F_{k+1,\ell}}i H_k (b_{\mathrm{d}}^{k+1,Q})(x)\zeta(x).$$
Hence, by \cite[Lemma 3.6]{SS2018},  we heve
\begin{equation}\label{bd-H-2}
\varphi(\chi_{(\lambda,\infty)}(|\zeta\widetilde{S}_n(b_{\mathrm{d}})\zeta|))\leq\varphi(\chi_{(\frac{\lambda}{2},\infty)}(|A|)+\varphi(\chi_{(\frac{\lambda}{2},\infty)}(|B|)).
\end{equation}

We now estimate the first term on the right-hand side.  Fix a point $t^Q\in Q$,
applying the cancellation $\int_{Q} b_{\mathrm{d}}^{k+1,\ell,Q}(t)r_k^{\alpha_k}(t)d\mu(t)=0$ from Theorem~\ref{NewCZ}\,(iv) and the definition of $H_k$ in \eqref{Hk}, we may rewrite $A$ as
$$
A(x)=\zeta(x)\sum_{k\geq0}\sum_{\ell=1}^{\lceil\log_2 m_k\rceil}\sum_{\substack{Q\in F_{k+1,\ell}\\ Q\subset I_k(x)}} \frac{M_k i r_k^{-\alpha_k}(x)}{2}
\int_{Q}b_{\mathrm{d}}^{k+1,\ell,Q}(t)\,r_k^{\alpha_k}(t)\times\Delta_{Q}(x,t,t^Q)\,d\mu(t)\,\zeta(x),$$
where the definition of $\Delta_{Q}(x,t,t^Q)$  is given in \eqref{DQ},  and following the discussion in Subcase~2.3, the integration domain is $Q\cap I_k(x)\cap \{x_k\neq t_k\}=Q$. Thus,
\begin{align*}
	\|A\|_{L_1(\N)}&\leq\frac{M_k}{2}\sum_{k\geq0}\sum_{\ell=1}^{\lceil\log_2 m_k\rceil}\sum_{I\in F_k'} \int_{(3Q)^c\cap I}\sum_{\substack{Q\in F_{k+1,\ell}\\ Q\subset I_k(x)}} \|\int_{Q}b_{\mathrm{d}}^{k+1,\ell,Q}(t)r_k^{\alpha_k}(t)\times\Delta_{Q}(x,t,t^Q)\,d\mu(t)\|_{L_1(\mathcal{M})}d\mu(x)\\
	&=\frac{M_k}{2}\sum_{k\geq0}\sum_{\ell=1}^{\lceil\log_2 m_k\rceil}\sum_{I\in F_k}\sum_{\substack{Q\in F_{k+1,\ell}\\ Q\subset I}} \int_{(3Q)^c\cap I} \|\int_{Q}b_{\mathrm{d}}^{k+1,\ell,Q}(t)r_k^{\alpha_k}(t)\times\Delta_{Q}(x,t,t^Q)\,d\mu(t)\|_{L_1(\mathcal{M})}d\mu(x).
	\end{align*}
For each $I\in F_k'$, $Q\in F_{k+1,\ell}$ and $Q\subset I$,  by triangle inequality, Fubini's theorem and  Lemma \ref{cot} (i),
	\begin{align*}
	&M_k\int_{(3Q)^c\cap I}\Big\|\int_{Q}|b_{\mathrm{d}}^{k+1,\ell,Q}(t)|\cdot|\Delta_{Q}(x,t,t^Q)|d\mu(t)\Big\|_{L_1(\mathcal{M})}d\mu(x)\\
	&\leq\int_{Q}\|b_{\mathrm{d}}^{k+1,\ell,Q}(t)\|_{L_1(\mathcal{M})}M_k\int_{(3Q)^c\cap I} \Big|\Delta_{Q}(x,t,t^Q)\Big|d\mu(x)d\mu(t)\\
	&\leq c \|b_{\mathrm{d}}^{k+1,\ell,Q}(t)\|_{L_1(\mathcal{N})}.
\end{align*}
Consequently, Theorem~\ref{NewCZ}\,(iv) yields
$$
	\|A\|_{L_1(\N)}\leq c\sum_{k\geq0}\sum_{\ell=1}^{\lceil\log_2 m_k\rceil}\sum_{Q\in F_{k+1,\ell}}\|b_{\mathrm{d}}^{k+1,\ell,Q}\|_{L_1(\mathcal{N})}\leq c\|f\|_{L_1(\mathcal{N})}.
$$
Therefore, by Chebyshev's inequality,
\[
\varphi\Bigl(\chi_{(\lambda/2,\infty)}(|A|)\Bigr)
\leq \frac{c\|f\|_{L_1(\mathcal{N})}}{\lambda}.
\]
The term $B$ is handled in the same way, using the cancellation $\int_Q b_{\mathrm d}^{k+1,\ell,Q}(t)\,d\mu(t)=0$  instead of  the cancellation $\int_{Q} b_{\mathrm{d}}^{k+1,\ell,Q}(t)r_k^{\alpha_k}(t)d\mu(t)=0$, which gives
\[
\varphi\Bigl(\chi_{(\lambda/2,\infty)}\Bigl(|B|\Bigr)
\leq \frac{c\|f\|_{L_1(\mathcal{N})}}{\lambda}.
\]
The proof is completed by combining the last two estimates with  \eqref{bd-H-2}.
\end{proof}

\begin{lemma}\label{boff-w11} There exists an absolute constant $c>0$ such that
	$$\varphi(\chi_{(\lambda,\infty)}(|\zeta\widetilde{S}_n(b_{\mathrm{off}})\zeta|))\leq\frac{c\|f\|_{L_1(\mathcal{N})}}{\lambda},\quad \forall n\geq1,$$
	{where \(\zeta \) is the projection defined in \eqref{zeta}.}
\end{lemma}
\begin{proof}  Fix $
	n=\sum_{k=0}^\infty \alpha_kM_k$ with $0\le \alpha_k<m_k$. Without loss of generality, assume that $f\in L_1^+(\mathcal{N})$. 
	Using the cancellation provided in Theorem \ref{NewCZ} (v),  and following the argument at the beginning of Lemma \ref{bd-w11} and using claim \eqref{claim}, we obtain
	\begin{equation}\label{boff-reduce-H}	
		\zeta\widetilde{S}_n(b_{\mathrm{off}})\zeta=A-B,
	\end{equation}
where for $x\in(3Q)^c$,
$$A(x)=\zeta(x)\sum_{k\geq0}\sum_{\ell=1}^{\lceil\log_2 m_k\rceil}\sum_{Q\in F_{k+1,\ell}}i r_k^{-\alpha_k}(x) H_k(b_{\mathrm{off}}^{k+1,\ell,Q}r_k^{\alpha_k})(x) \zeta(x)$$
and
$$B(x)=\zeta(x)\sum_{k\geq0}\sum_{\ell=1}^{\lceil\log_2 m_k\rceil}\sum_{Q\in F_{k+1,\ell}}i H_k (b_{\mathrm{off}}^{k+1,Q})(x)\zeta(x).$$
Consequently, by \cite[Lemma 3.6]{SS2018} and  Chebyshev's inequality, 
\begin{equation}\label{boff-split}
	\begin{aligned}
	\varphi\Bigl(\chi_{(\lambda,\infty)}\bigl(|\zeta\,\widetilde{S}_n(b_{\mathrm{off}})\,\zeta|\bigr)\Bigr)
	&\le \varphi\Bigl(\chi_{(\lambda/2,\infty)}\Bigl(|A|\Bigr)+	\varphi\Bigl(\chi_{(\lambda/2,\infty)}\Bigl(|B|\Bigr)\\
	&\leq c\frac{\|A\|_{L_1(\mathcal N)}+\|B\|_{L_1(\mathcal N)}}{\lambda}.
	\end{aligned}
\end{equation}
Thus it suffices to prove
\begin{equation}\label{AB1}
\|A\|_{L_1(\mathcal N)}+\|B\|_{L_1(\mathcal N)}
\leq c
\|f\|_{L_1(\mathcal N)}.
\end{equation}
We now estimate $A$, the term $B$ is handled similarly.  Fix a point $t^Q\in Q$, 
as in Lemma~\ref{bd-w11}, we have
\[Q\cap I_k(x)\cap \{x_k\neq t_k\}=Q.
\]
By the cancellation $
\int_Q b_{\mathrm{off}}^{k+1,\ell,Q}(t)\,r_k^{\alpha_k}(t)\,d\mu(t)=0$ from Theorem~\ref{NewCZ}\,(v) and the definition of $H_k$ in \eqref{Hk}, we may rewrite
\begin{equation}\label{boff-Hk-diff}
	A(x)=\zeta(x)\sum_{k\geq0}\sum_{\ell=1}^{\lceil\log_2 m_k\rceil}\sum_{\substack{Q\in F_{k+1,\ell}\\ Q\subset I_k(x)}} \frac{M_k i r_k^{-\alpha_k}(x)}{2}
	\int_{Q}b_{\mathrm{off}}^{k+1,\ell,Q}(t)\,r_k^{\alpha_k}(t)\times\Delta_{Q}(x,t,t^Q)\,d\mu(t)\,\zeta(x),
\end{equation}
where we refer to \eqref{DQ} for the definition of $\Delta_{Q}(x,t,t^Q)$.
Using symmetry and the triangle inequality, we obtain
\begin{equation}\label{A-l1}
\begin{aligned}
	\|A\|_{L_1(\mathcal{N})}
	&\le \frac{M_k}{2}\sum_{k\geq0}\sum_{\ell=1}^{\lceil\log_2 m_k\rceil}\sum_{I\in F_k}\int_{(3Q)^c\cap I}\|\sum_{\substack{Q\in F_{k+1,\ell}\\ Q\subset I_k(x)}}\int_{Q}b_{\mathrm{off}}^{k+1,\ell,Q}(t)\,r_k^{\alpha_k}(t)\times\Delta_{Q}(x,t,t^Q)\,d\mu(t)\|_{L_1(\mathcal{M})}d\mu(x)\\
	&\le \sum_{k\geq0}\sum_{\ell=1}^{\lceil\log_2 m_k\rceil}\sum_{I\in F_k'}\sum_{\substack{Q\in F_{k+1,\ell}\\ Q\subset I}}A_{1}^{k,\ell,I,Q}+\sum_{k\geq0}\sum_{\ell=1}^{\lceil\log_2 m_k\rceil}\sum_{I\in F_k'}\sum_{\substack{Q\in F_{k+1,\ell}\\ Q\subset I}}A_2^{k,\ell,I,Q},
\end{aligned}
\end{equation}
where 
$$A_1^{k,\ell,I,Q}=M_k\int_{(3Q)^c\cap I}\|\int_{Q}p_Qf(t)q_Q\,r_k^{\alpha_k}(t)\times\Delta_{Q}(x,t,t^Q)\,d\mu(t)\|_{L_1(\mathcal{M})}d\mu(x)$$
and 
$$A_2^{k,\ell,I,Q}=M_k\int_{(3Q)^c\cap I}\|\int_{Q}p_QA_Q(f)(t)q_Q\,r_k^{\alpha_k}(t)\times\Delta_{Q}(x,t,t^Q)\,d\mu(t)\|_{L_1(\mathcal{M})}d\mu(x).$$

The two estimates  are similar, so we first deal with the common argument. For each $I\in F_k'$, $Q\in F_{k+1,\ell}$ and $Q\subset I$, let \(h\) be a positive \(\mathcal M\)-valued function on \(Q\), and define
	\[
\mathcal T_h^{k,\ell,I,Q}
:=
M_k\int_{(3Q)^c\cap I}
\Big\|
\int_Q p_Qh(t)q_Q\,r_k^{\alpha_k}(t)\,\Delta_Q(x,t,t^Q)\,d\mu(t)
\Big\|_{L_1(\mathcal M)}\,d\mu(x).
\]
By Lemma~\ref{AcB}, there exists a contraction \(u_{h,Q}\in \mathcal M\) such that
\[
\int_Q p_Qh(t)q_Q\,r_k^{\alpha_k}(t)\,\Delta_Q(x,t,t^Q)\,d\mu(t)
=B_{1,h}(x)^{1/2}\,u_{h,Q}(x)\,B_{2,h}^{1/2}.
\]
where
$$B_{1,h}(x)=\int_Q p_Q h(t)p_Q |\Delta_{Q}(x,t,t^Q)|^2\,d\mu(t)\quad
\mbox{and}\quad B_{2,h}=\int_Q q_Q h(t)q_Q\,d\mu(t)$$
Hence
\[
\mathcal T_h^{k,\ell,I,Q}
\le
M_k\|B_{2,h}^{1/2}\|_{L_\infty(\mathcal M)}
\int_{(3Q)^c\cap I}\|B_{1,h}(x)^{1/2}\|_{L_1(\mathcal M)}\,d\mu(x).
\]
Decompose $(3Q)^c\cap I$ into level-$(k+1)$ atoms $J\subset I\cap(3Q)^c$, where $|J|=M_{k+1}^{-1}$.
For each such $J$,  H\"older's inequality and Fubini's theorem give
\begin{align*}
	\int_{J}\bigl\|B_{1,h}(x)^\frac{1}{2}\bigr\|_{L_1(\mathcal{M})}\,d\mu(x)&\leq\|p_Q\|_{L_1(\mathcal{M})}^{\frac 12}\int_J\big\|B_{1,h}(x)\big\|_{L_1(\mathcal{M})}^\frac{1}{2} d\mu(x)\\
	&\leq\|p_Q\|_{L_1(\mathcal{M})}^{\frac 12} |J|^{\frac{1}{2}} \Big(\int_{Q}\Big\|p_Qh(t)p_Q\Big\|_{L_1(\mathcal{M})}\int_{J}\Big|\Delta_{Q}(x,t,t^Q)\Big|^2d\mu(x)d\mu(t)\Big)^\frac{1}{2}.
\end{align*}
Applying Lemma~\ref{cot} (ii), summing over all such atoms $J$ and  following the argument for the second term estimate in Lemma \ref{cot} (i), we obtain
\begin{equation}\label{common-estimate}
	\begin{aligned}
		\mathcal T_h^{k,\ell,I,Q}
	&\le cM_kM_{k+1}^{-\frac 12}\frac{c\,m_k^{1/2}}{M_k^{1/2}}\sum_{a=d}^{m_k}
	\frac{d}{a(a+d)}\|B_{2,h}^{1/2}\|_{L_\infty(\mathcal M)}\|p_Q\|_{L_1(\mathcal{M})}^{\frac 12}\Big(
	\int_Q \|p_Qh(t)p_Q\|_{L_1(\mathcal M)}\,d\mu(t)
	\Big)^{1/2}\\
	&\leq c\,\|B_{2,h}^{1/2}\|_{L_\infty(\mathcal M)}\,
	\|p_Q\|_{L_1(\mathcal M)}^{1/2}
	\Big(
	\int_Q \|p_Qh(t)p_Q\|_{L_1(\mathcal M)}\,d\mu(t)
	\Big)^{1/2}.
	\end{aligned}
\end{equation}

We now apply \eqref{common-estimate} to \(A_1^{k,\ell,I,Q}\) and
\(A_2^{k,\ell,I,Q}\).

\smallskip
\noindent\textbf{Estimate of \(A_1^{k,\ell,I,Q}\).}
Here \(h=f\). By Lemma~\ref{Cuculescu}(iii),
\[
\|B_{2,f}^{1/2}\|_{L_\infty(\mathcal M)}
=
\big\|
\Big(
\int_Q q_Qf(t)q_Q\,d\mu(t)
\Big)^{1/2}
\big\|_{L_\infty(\mathcal M)}
\le
(|Q|\lambda)^{1/2}.
\]
Also,
\[
\int_Q \|p_Qf(t)p_Q\|_{L_1(\mathcal M)}\,d\mu(t)
=
\varphi(p_Qfp_Q\,\chi_Q).
\]
Therefore, \eqref{common-estimate} yields
\begin{equation}\label{a1klq}
	A_1^{k,\ell,I,Q}
	\le
	c\,\lambda^{1/2}\,
	\varphi(p_Q\chi_Q)^{1/2}\,
	\varphi(p_Qfp_Q\,\chi_Q)^{1/2}.
\end{equation}

\smallskip
\noindent\textbf{Estimate of \(A_2^{k,\ell,I,Q}\).} The only difference here is that \(A_Q(f)\) is not positive. We therefore decompose \(A_Q(f)\) into real and imaginary parts, and then into positive and negative parts:
\begin{align*}
A_Q(f)&=\mathrm{Re}A_Q(f)+i\,\mathrm{Im} A_Q(f)\\
&=(\mathrm{Re} A_Q(f))^+-(\mathrm{Re} A_Q(f))^-+i(\mathrm{Im} A_Q(f))^+-i(\mathrm{Im} A_Q(f))^-=:\sum_{\nu=1}^4A_Q^\nu(f),
\end{align*}
by  Lemma~\ref{KAQ},
\begin{equation}\label{GammaQ-bd}
	0\le A_Q^\nu(f)(t)\le c\,f_Q,
	\qquad
	t\in Q,\quad \nu=1,2,3,4.
\end{equation}
Hence, by the triangle inequality,
\[
A_2^{k,\ell,I,Q}
\le
\sum_{\nu=1}^4 \mathcal T_{A_Q^\nu(f)}^{k,\ell,I,Q}.
\]
Applying \eqref{common-estimate} with \(h=A_Q^\nu(f)\), and using
\eqref{GammaQ-bd} together with Lemma~\ref{Cuculescu}(iii), we get
\[
\|B_{2,A_Q^\nu(f)}^{1/2}\|_{L_\infty(\mathcal M)}=\|(\int_Q q_Q A_Q^\nu(f)q_Q\,d\mu(t))^{1/2}\|_{L_\infty(\mathcal M)}
\le
c\,(|Q|\lambda)^{1/2},
\]
and
\[
\int_Q \|p_QA_Q^\nu(f)(t)p_Q\|_{L_1(\mathcal M)}\,d\mu(t)
\le
c\,\varphi(p_Qfp_Q\,\chi_Q).
\]
Therefore,
\[
\mathcal T_{A_Q^\nu(f)}^{k,\ell,I,Q}
\le
c\,\lambda^{1/2}\,
\varphi(p_Q\chi_Q)^{1/2}\,
\varphi(p_Qfp_Q\,\chi_Q)^{1/2},
\]
and summing over \(\nu\) yields
\begin{equation}\label{a2klq}
	A_2^{k,\ell,I,Q}
	\le
	c\,\lambda^{1/2}\,
	\varphi(p_Q\chi_Q)^{1/2}\,
	\varphi(p_Qfp_Q\,\chi_Q)^{1/2}.
\end{equation}
Substituting \eqref{a1klq} and \eqref{a2klq} into \eqref{A-l1}, and then
applying Cauchy-Schwarz inequality, we obtain
\begin{align*}
	\|A\|_{L_1(\mathcal N)}
	&\le
	c\sum_{k\ge 0}\sum_{\ell=1}^{\lceil\log_2 m_k\rceil}
	\sum_{I\in F_k'}\sum_{\substack{Q\in F_{k+1,\ell}\\ Q\subset I}}
	\lambda^{1/2}\,
	\varphi(p_Q\chi_Q)^{1/2}\,
	\varphi(p_Qfp_Q\,\chi_Q)^{1/2} \\
	&\le
	c\Biggl(
	\sum_{k\ge 0}\sum_{\ell=1}^{\lceil\log_2 m_k\rceil}\sum_{Q\in F_{k+1,\ell}}
	\lambda\,\varphi(p_Q\chi_Q)
	\Biggr)^{1/2}
	\Biggl(
	\sum_{k\ge 0}\sum_{\ell=1}^{\lceil\log_2 m_k\rceil}\sum_{Q\in F_{k+1,\ell}}
	\varphi(p_Qfp_Q\,\chi_Q)
	\Biggr)^{1/2} \\
	&\le
	c\,(\lambda\varphi(1-q))^{1/2}\,\bigl(\varphi((1-q)f)\bigr)^{1/2}
	\le
	c\,\|f\|_{L_1(\mathcal N)},
\end{align*}
where  the last inequality follows from Lemma~\ref{Cuculescu}(iv).

The term \(B\) is treated exactly in  the same way, using the cancellation$\int_{Q} b_{\mathrm{d}}^{k+1,\ell,Q}(t)d\mu(t)=0$ instead of
$\int_{Q} b_{\mathrm{d}}^{k+1,\ell,Q}(t)r_k^{\alpha_k}(t)d\mu(t)=0.$
Thus
\[
\|B\|_{L_1(\mathcal N)}\le c\,\|f\|_{L_1(\mathcal N)}.
\]
This proves \eqref{AB1}. The proof is complete.
\end{proof}

\begin{proof}[Proof of Proposition \ref{hatS}] 	It is enough to prove the uniform weak type \((1,1)\) estimate
\begin{equation}\label{w11-eq}
		\|\widetilde S_n f\|_{L_{1,\infty}(\mathcal N)}
	\le c\|f\|_{L_1(\mathcal N)},
	\qquad n\ge 1.
\end{equation}
	Indeed, once  \eqref{w11-eq} is established, Lemma \ref{lem:type22} together with the
Marcinkiewicz interpolation theorem \cite{Be1988} and a standard duality argument
imply that $\widetilde{S}_n$ is of strong type $(p,p)$ for all $1<p<\infty$.  

We now prove \eqref{w11-eq}. Without loss of generality, assume that $f\in L_1^+(\mathcal{N})$. For any $\lambda>0$, and  let  
	$$f=g_{\mathrm{d}}+g_{\mathrm{off}}+b_{\mathrm{d}}+b_{\mathrm{off}}$$
be the decomposition given by Theorem~\ref{NewCZ}. We write
\begin{align*}
\widetilde{S}_n(f)&=\widetilde{S}_n(g_{\mathrm{d}})+\widetilde{S}_n(g_{\mathrm{off}})+(1-\zeta)\widetilde{S}_n(b_{\mathrm{d}})+\zeta\widetilde{S}_n(b_{\mathrm{d}})(1-\zeta)+\zeta\widetilde{S}_n(b_{\mathrm{d}})\zeta\\
&\quad+(1-\zeta)\widetilde{S}_n(b_{\mathrm{off}})+\zeta\widetilde{S}_n(b_{\mathrm{off}})(1-\zeta)+\zeta\widetilde{S}_n(b_{\mathrm{off}})\zeta,
\end{align*}
where \(\zeta\) is  the projection defined in \eqref{zeta}.
By \cite[Lemma 3.6]{SS2018} and Theorem~\ref{NewCZ}\,(i), we have
\begin{align*}
	\varphi\bigl(\chi_{(8\lambda,\infty)}(|\widetilde{S}_n f|)\bigr)
	&\le \varphi\bigl(\chi_{(\lambda,\infty)}(|\widetilde{S}_n(g_{\mathrm{d}})|)\bigr)
	+\varphi\bigl(\chi_{(\lambda,\infty)}(|\widetilde{S}_n(g_{\mathrm{off}})|)\bigr)\\
	&\quad+\varphi\bigl(\chi_{(\lambda,\infty)}(|\zeta\widetilde{S}_n(b_{\mathrm{d}})\zeta|)\bigr)
	+\varphi\bigl(\chi_{(\lambda,\infty)}(|\zeta\widetilde{S}_n(b_{\mathrm{off}})\zeta|)\bigr)
	+\frac{c}{\lambda}\,\|f\|_{L_1(\mathcal{N})}.
\end{align*}
We first estimate the good diagonal term. By Chebyshev's inequality,
Lemma~\ref{lem:type22}, and Theorem~\ref{NewCZ} (ii),
\begin{align*}
\varphi\{|\widetilde{S}_n(g_{\mathrm{d}})|>\lambda\}&\leq \frac {1}{\lambda^2} \|\widetilde{S}_n(g_{\mathrm{d}})\|_{L_2(\mathcal{N})}^2\leq \frac {1}{\lambda^2} \|g_{\mathrm{d}}\|_{L_2(\mathcal{N})}^2\\
&\leq \frac{1}{\lambda^2} \|g_{\mathrm{d}}\|_{L_\infty(\mathcal{N})}\|g_{\mathrm{d}}\|_{L_1(\mathcal{N})}\leq  \frac{c}{\lambda}\|f\|_{L_1(\mathcal{N})}.
\end{align*}
For the good off-diagonal term, Chebyshev's inequality and Theorem~\ref{NewCZ} (iii) give directly
$$
	\varphi\{|\widetilde{S}_n(g_{\mathrm{off}})|>\lambda\}\leq \frac {1}{\lambda^2} \|\widetilde{S}_n(g_{\mathrm{off}})\|_{L_2(\mathcal{N})}^2\leq \frac {1}{\lambda^2} \|g_{\mathrm{off}}\|_{L_2(\mathcal{N})}^2\leq  \frac{c}{\lambda}\|f\|_{L_1(\mathcal{N})}.
$$
Lemmas~\ref{bd-w11} and~\ref{boff-w11}, together with the two estimates above, now yield \[ \varphi\bigl(\chi_{(8\lambda,\infty)}(|\widetilde S_n f|)\bigr) \le \frac{c}{\lambda}\|f\|_{L_1(\mathcal N)}. \] This is exactly \eqref{w11-eq}, and the proof is complete.
\end{proof}

\section{Proof of Theorem \ref{Npb}}\label{sec5}
 In this section, we introduce the noncommutative Vilenkin system of hyperfinite II$_1$ factor $\mathcal R$ and prove Theorem~\ref{Npb}. The construction of the Vilenkin system goes back to \cite[Section 4]{DFdePS2001},  \cite[Page 90]{DS2000} and \cite{AFS}.

Let $\mathcal R$ be the hyperfinite type II$_1$ factor (see \cite{AFS}), and let $(m_k)_{k\geq0}\subset \mathbb N$ satisfy $m_k\geq2$ for every $k\geq0$.
Denote by $\mathbb M_n$ the algebra of $n\times n$ complex valued matrices, equipped with its
normalized trace  ${\rm tr}_n$  satisfying ${\rm tr}_n(1_n)=1$, where $1_n$ is the $n\times n$ identity matrix. We identify $\mathcal R$ with the infinite tensor product
$$(\mathcal R,\tau)=\bigotimes_{k=0}^\infty (\mathbb M_{m_k}, {\rm tr}_{m_k}).$$
Note that such $\tau $ is a faithful normal trace on $\mathcal R$. The construction of infinite tensor products is detailed in \cite[Definition XIV.1.6]{Ta2003}. Consider the von Neumann subalgebras of $\mathcal R$ defined by setting
\begin{equation}\label{f-Rn}
	(\mathcal R_n,\tau_n)=\bigotimes_{k=0}^{n-1} (\mathbb M_{m_k}, {\rm tr}_{m_k}).
\end{equation}
Usually, we write $\mathcal R_0=\mathbb C$. In fact, we view $\mathcal R_n$ as a von Neumann subalgebra of $\mathcal R_{n+1}$ (resp. $\mathcal R$) via the inclusion
$$x\in \mathcal R_n\mapsto x \otimes 1_{m_{n}}\in \mathcal R_{n+1}\quad  \big(\mbox{resp.}\quad x \otimes \big(\bigotimes_{k=n}^\infty1_{m_k}\big)\in \mathcal R\big)  .$$
The conditional expectation $\mathcal E_n:\mathcal R\rightarrow \mathcal R_n$ is given by
$$\mathcal E_n =\big(\bigotimes_{k=0}^{n-1} 1_{m_k}\big)\otimes \big(\bigotimes_{k=n}^\infty {\rm tr}_{m_k}\big).$$

Given a  sequence $\mathbf m=(m_k)_{k\geq0}$, let $G_{\mathbf m}=\prod_{k=0}^{\infty}\mathbb Z_{m_k}$ be the associated Vilenkin group. Its dual group  $\widehat{G}_{m}$ is given by
$$\widehat{G}_{\mathbf{m}}=\coprod_{k=0}^{\infty} \widehat{\mathbb{Z}}_{m_k},$$
where each $\widehat{\mathbb Z}_{m_k}$ is canonically isomorphic to $\mathbb Z_{m_k}$.
For more properties of the Vilenkin group, refer to \cite{Golubov1991}

The following details the construction of the corresponding Vilenkin system in $\mathcal{R}$. For each integer $k \geq 0$, let $A_{k}$ and $B_{k}$ denote the matrices given by
$$
A_{k}=\sum_{j \in \mathbb{Z}_{m_k}} \exp\big(\frac{2 \pi \mathrm i j}{m_k}\big) e_{j, j}^{m_k}, \quad B_{k}=\sum_{j \in \mathbb{Z}_{m_k}} e_{j, j+1}^{m_k},
$$
where for each $n \geq 1, e_{p, q}^{n}$ is the $n \times n$  matrix defined such that $(p,q)$-th entry equal to $1$, and all other entries are zero, with entries indexed by $\mathbb{Z}_{n}^{2}$.
Within $\mathcal{R}$, the elements of the Vilenkin system are enumerated by doubled sequences in $\widehat{G}_{\mathbf{m}}$.  We shall identify the group $G_{\mathbf{m}}\times G_{\mathbf{m}}$ (respectively, $\widehat{G}_{\mathbf{m}}\times \widehat{G}_{\mathbf{m}}$) with the Vilenkin group $G_{\mathbf{2m}}$ (respectively, $\widehat{G}_{\mathbf{2m}}$) where
\begin{equation}\label{2m}
\mathbf{2m}:=\{(\mathbf{2m})_k\}_{k\geq0},
\end{equation}
where $(\mathbf{2m})_{2k}=(\mathbf{2m})_{2k+1}=m_k$ for each $k\geq0.$
It follows that any $\eta\in \widehat{G}_{\mathbf{2m}}=\coprod_{k=0}^{\infty} \widehat{\mathbb{Z}}_{(\mathbf{2m})_k}\cong \prod_{k=0}^{\infty}\mathbb Z_{(\mathbf{2m})_k}$ can be considered as a pair $(\eta',\eta'')$ where
$$\eta':=(\eta_{2k})_{k\geq0}\quad \mbox{and \quad $\eta'':=(\eta_{2k+1})_{k\geq0}$}$$
are elements of $\widehat{G}_{\mathbf{m}}$. For each $\eta \in \widehat{G}_{\mathbf{2m}}$, define
\begin{align}\label{non-V-S}
	V_{\eta}=\bigotimes_{k=0}^{\infty} A_{k}^{\eta_{2 k}} B_{k}^{\eta_{2 k+1}}.
\end{align}
The Vilenkin system associated to $G_{\mathbf{2m}}$ in $\mathcal{R}$ is then given by the set $\{V_{\eta}\}_{\eta \in \widehat{G}_{\mathbf{2m}}}$.  It is easily verified that every element of the Vilenkin system is a unitary operator.
For doubled sequences, let $G_{2 N}$ denote the subgroup of sequences of length $2 N$ in $G_{\mathbf{2m}}$, and $\widehat{G}_{2 N}$ denote the subgroup of sequences of length $2 N$ in $\widehat{G}_{\mathbf{2m}}$. It is then easily verified that for each $N \geq 1$, the restriction of the Vilenkin system to $\{V_{\eta}\}_{\eta \in \widehat{G}_{2 N}}$ forms a basis for $\mathcal{R}_{N}$ and $L_{2}(\mathcal{R}_{N})$.  Since the mapping 
$$k:\eta\in  \widehat{G}_{\mathbf{2m}}\mapsto k_{\eta}=\sum_{k=0}^\infty\eta_kM_{k}\in\mathbb{N}$$
is an order isomorphism of $\widehat{G}_{\mathbf{2m}}$  and $\mathbb{N}=\{0,1,2, \ldots\}$,
where $M_0=1$, $M_{1}=(\mathbf{2m})_0M_0$ and $M_{k+1}=(\mathbf{2m})_{k}M_{k}$ for each $k\geq1$.
Denote $V_n:=V_{k^{-1}(n)}$. Then the Vilenkin system $\{V_n\}_{0\leq n \leq M_{2N}-1}$ also forms an orthonormal basis for $\mathcal{R}_{N}$ and $L_{2}(\mathcal{R}_{N})$.  

The noncommutative Vilenkin Fourier coefficient of $f \in L_1(\mathcal{R})$ is defined by  $\hat f(0) = \tau (f)$
and
$$\widehat f(k)=\tau(fV^{\ast}_{k}), \quad k\geq1.$$
Similar to \eqref{ps}, define the partial sum  of $f\in L_1(\mathcal{R})$
\begin{equation}\label{nps}
	\mathcal{S}_{n}(f) := \sum_{k=0}^{n-1} \widehat {f}(k)V_{k},\quad n\geq1.
\end{equation}

The following lemma is taken from \cite[Lemma 2.1]{SS2018}. 
\begin{lemma}\label{RW-plus} For any $\ell,n\geq1$, there exist complex numbers   $u_n$ and $\omega_{\ell,n}$ such that $V_n^*=u_nV_{\bar{n}}$ and $V_\ell V_n=\omega_{\ell,n}V_{\ell\dot{+} n}$, where $|u_n|=1$, $|\omega_{\ell,n}|=1$ and $\ell\dot{+}n $ is referred to \eqref{addition}.
\end{lemma}

A polynomial  in $\mathcal{R}$ is a finite sum 
$$f=\sum_{k\geq0}\beta_kV_k, \quad \beta_k\in \mathbb{C},$$
that is $\beta_k=0$ for all but finite indices $k\geq0$. Let $\mathrm{Poly}(\Ra)$ be  the $*$-subalgebra of $\mathcal{R}$ consisting with all (complex) polynomials in $\mathcal{R}$. 
The von Neumann algebra  $L_{\infty}(0,1)\bar{\otimes} \mathcal{R}$ is equipped with the trace $\int\otimes \tau$ which is still denoted by $\varphi$.  Define the map $\gamma:\mathrm{Poly}(\Ra)\to L_{\infty}(0,1)\bar{\otimes} \mathcal{R}$ by 
\begin{equation}\label{transference mapping}
	\gamma(f)=\sum_{k\geq0}\widehat{f}(k)\phi_{k}\otimes V_{k}, \quad f\in \mathrm{Poly}(\Ra),
\end{equation}
where  $(\phi_k)_{k\geq0}$ are the classical Vilenkin functions with respect to the Vilenkin group $G_{\mathbf{2m}}$.

Since $\mathrm{Poly}(\Ra)$ is weak-$\ast$ dense in $ \Ra $, the map $\gamma:\mathrm{Poly}(\Ra)\to L_{\infty}(0,1)\bar{\otimes} \mathcal{R}$ is normal, we could now extend $\gamma$ in a natural way making $\gamma$ a $*$-homomorphism from $\mathcal{R}$ into $L_{\infty}(0,1)\bar{\otimes} \mathcal{R}$. For simplicity, we still use the notation $\gamma$ to stand for its natural extension if no confusion arises.
\begin{lemma}[{\cite[Proposition 3.2, Proposition 3.3]{JLZZ2025}}]\label{transference 1}
	The mapping $\gamma:\Ra\to L_{\infty}(0,1)\bar{\otimes}\mathcal{R}$ satisfies the following properties:
	\begin{enumerate}[{\rm (i)}]
		\item $\gamma$ is a $*$-homomorphism;
		\item $\gamma$ is trace preserving, that is, $\varphi(\gamma(f))=\tau(f)$, for all $f\in \mathrm{Poly}(\Ra)$;
		\item $\gamma$ is injective, and both $\gamma$ and $\gamma^{-1}$ are normal;
		\item $\gamma(\mathcal{R})$ is a von Neumann subalgebra of $L_{\infty}(0,1)\bar{\otimes} \mathcal{R}$.
	\end{enumerate}
\end{lemma}
\begin{lemma}[{\cite[Theorem 3.4]{JLZZ2025}}]\label{transference2}
	For any $0<p\leq \infty$ and $ f\in\mathcal{R}$, we have $$\|\gamma(f)\|_{L_{p}(L_{\infty}(0,1)\bar{\otimes} \mathcal{R})}=\|f\|_{L_{p}(\mathcal{R})}.$$ Furthermore, for any $f\in \mathcal{R},$ one has the following stronger result
	$$\mu(t,\gamma (f))= \mu(t, f),\quad \forall t>0$$
	and
$$\|\gamma(f)\|_{L_{1,\infty}(L_{\infty}(0,1)\bar{\otimes} \mathcal{R})}=\|f\|_{L_{1,\infty}(\mathcal{R})}.$$
\end{lemma}

Now we are ready to prove Theorem \ref{Npb} in detail. 
\begin{proof}[Proof of Theorem \ref{Npb}] 	Let $	\mathcal N=L_\infty(0,1)\bar\otimes \mathcal R$ and  $f\in L_{1}(\mathcal{R})$. Without loss of generality, we assume that $f$ is positive. Since $\gamma$ is a   $*$-homomorphism (see Proposition \ref{transference 1}), $\gamma(f)\geq0$.   It follows directly from
\eqref{nps}  and	\eqref{transference mapping} that
	\begin{align*}
		\gamma (\mathcal{S}_n(f))&=S_{n}(\gamma(f)),
	\end{align*}	
	where $S_{n}$ is as in \eqref{ps}. Then, by Lemma \ref{transference2}, we have
	\begin{align}\label{1-equality}
		\|\mathcal{S}_n(f)\|_{L_{1,\infty}(\mathcal{R})}= \|\gamma(\mathcal{S}_n(f))\|_{L_{1,\infty}(\mathcal{N})}=\|S_n(\gamma(f)\|_{L_{1,\infty}(\mathcal{N})}.
	\end{align}
	Combining Theorem \ref{weak-main} and Lemma  \ref{transference2}, we arrive at
$$	\Big\|(S_n(\gamma(f))\Big\|_{L_{1,\infty}(\mathcal{N})}
		\leq c\|\gamma(f)\|_{L_{1}(\mathcal{N})}=c\|f\|_{L_1(\mathcal{R})},
$$
 Then, combining \eqref{1-equality}, we get \eqref{nc-w11} of Theorem \ref{Npb}.
	
Similarly, by the  Marcinkiewicz interpolation theorem \cite{Be1988} and a duality argument, we can show the strong type $(p,p)$ inequality \eqref{nc-spp}, and the proof is complete. 
\end{proof}

\noindent \textbf{Acknowledgements}
The authors are very grateful to Fedor Sukochev for his helpful comments and suggestions on this paper. 

%
%
%
%
%
%
%
%

%


\providecommand{\bysame}{\leavevmode\hbox to3em{\hrulefill}\thinspace}
\providecommand{\MR}{\relax\ifhmode\unskip\space\fi MR }
\providecommand{\MRhref}[2]{%
	\href{http://www.ams.org/mathscinet-getitem?mr=#1}{#2}
}
\providecommand{\href}[2]{#2}

\end{document}